\documentclass{amsart} 

\usepackage[english]{babel}
\usepackage[latin1]{inputenc}
\usepackage[autostyle, english = american]{csquotes} 

\usepackage{amsmath}
\usepackage{amsthm}
\usepackage{amssymb}
\usepackage{tikz}

\usepackage{imakeidx} 
\usepackage{appendix}
\usepackage{tikz-cd}
\usepackage{verbatim}
\usepackage{enumitem}
\usepackage{multicol}

\usepackage[backref=page]{hyperref}
\usepackage{cleveref}

\hypersetup{colorlinks=true,linkcolor=blue,anchorcolor=blue,citecolor=blue}

\usepackage{adjustbox}

\usepackage{mathtools}

\usepackage{amsfonts}

\newtheorem{thm}{Theorem}[section]
\newtheorem{prop}[thm]{Proposition}
\newtheorem{lemma}[thm]{Lemma}

\theoremstyle{definition}

\theoremstyle{remark}
\newtheorem{rmk}[thm]{Remark}
\numberwithin{equation}{section} 

\usepackage{algorithm}
\usepackage{algorithmic}

\usepackage{relsize}
\usepackage[bbgreekl]{mathbbol}
\DeclareSymbolFontAlphabet{\mathbb}{AMSb} 
\DeclareSymbolFontAlphabet{\mathbbl}{bbold}

\newcommand{\Q}{\mathbb Q}
\newcommand{\F}{\mathbb F}

\newcommand{\Z}{\mathbb Z}

\newcommand{\cl}{\overline}

\renewcommand{\phi}{\varphi}

\newcommand{\on}[1]{\operatorname{#1}}
\newcommand{\ang}[1]{\left \langle{#1}\right \rangle}

\title[The lifting problem for Galois representations]{The lifting problem for Galois representations}

\author{Alexander Merkurjev}
\address{Department of Mathematics\\
	University of California\\
	Los Angeles, CA 90095 \\United States of America}
\email{merkurev@math.ucla.edu}

\author{Federico Scavia}
\address{CNRS\\
	Institut Galil\'ee\\
	Universit\'e Sorbonne Paris Nord\\
	99 avenue Jean-Baptiste Cl\'ement, 93430\\ 
	Villetaneuse, France}
\email{scavia@math.univ-paris13.fr}

\makeatletter
\@namedef{subjclassname@2020}{\textup{2020} Mathematics Subject Classification}
\makeatother


\begin{document}
	
	\begin{abstract}
We solve the lifting problem for Galois representations in every dimension and in every characteristic. That is, we determine all pairs $(n,k)$, where $n$ is a positive integer and $k$ is a field of characteristic $p>0$, such that for every field $F$, every continuous homomorphism $\Gamma_F\to \mathrm{GL}_n(k)$ lifts to $\mathrm{GL}_n(W_2(k))$, where $\Gamma_F$ is the absolute Galois group of $F$ and $W_2(k)$ is the ring of $p$-typical length $2$ Witt vectors of $k$.	
\end{abstract}

	\maketitle
	
	\section{Introduction}
    
	\subsection{Lifting Galois representations} Let $F$ be a field, let $\Gamma_F$ be the absolute Galois group of $F$, let $k$ be a field of characteristic $p>0$, let $W_2(k)$ be the ring of $p$-typical length $2$ Witt vectors of $k$, and let $n$ be a positive integer. Given an $n$-dimensional continuous $k$-linear representation $V$ of $\Gamma_F$, a basic question is whether $V$ lifts to $W_2(k)$, that is, whether there exists a continuous $W_2(k)$-free $\Gamma_F$-module $W$ of rank $n$ such that $W\otimes_{W_2(k)}k\cong V$. Similarly, for an $n$-dimensional complete flag of continuous $\Gamma_F$-representations, that is, a sequence of continuous $\Gamma_F$-representations $V_1\subset V_2\subset\cdots \subset V_n$ such that $V_i$ has dimension $i$ for all $1\leq i\leq n$, one may ask whether the flag lifts to $W_2(k)$, that is, whether there exists a sequence of $W_2(k)$-free continuous $\Gamma_\F$-modules $W_1\subset W_2\subset\cdots \subset W_n$, such that $W_{i+1}/W_i$ is $W_2(k)$-free for all $1\leq i\leq n$, and which reduces to the sequence of the $V_i$ after tensorization with $k$ over $W_2(k)$. 
    
    The related question of existence of lifting representations of $\Gamma_{\Q}$ to characteristic zero, perhaps satisfying additional conditions, is of great importance in number theory. For example, given a continuous odd representation $\rho\colon \Gamma_{\Q}\to \on{GL}_2(\cl{\F}_p)$, it is very useful to construct a continuous lifting $\tilde{\rho}\colon \Gamma_{\Q} \to \on{GL}_2(\cl{\Q}_p)$ which is unramified outside finitely many places. The existence of such liftings, due to Ramakrishna \cite{ramakrishna1999lifting} and Khare--Wintenberger \cite{khare2009serre}, is a key tool in the proof Serre's modularity conjecture by Khare--Wintenberger \cite{khare2006serre, khare2009serre, khare2009serre-b}. More generally, the deformation theory of continuous representations of absolute Galois groups of local and global fields is a prominent topic in number theory, with connections to modularity theorems and Wiles' proof of Fermat's Last Theorem \cite{wiles1995modular}.

    \subsection{The question of Khare and Serre.} Khare \cite{khare1997base} proved that, when $k$ is a finite field, every $2$-dimensional continuous representation of $\Gamma_F$ with coefficients in $k$ lifts to $W_2(k)$, for every field $F$. More precisely, Khare stated his theorem in the case when $F$ is a number field, and Serre observed that Khare's argument generalized to an arbitrary field $F$; see \cite[Remark 2 p. 392]{khare1997base}. Khare and Serre then asked whether every continuous finite-dimensional representation of $\Gamma_F$ with coefficients in $k$ lifts to $W_2(k)$; see \cite[Question 1.1]{khare2020liftable}. 
    
    De Clercq and Florence \cite{declercq2022lifting} generalized Khare's theorem by removing the assumption that $k$ is finite (see Khare--Larsen \cite{khare2020liftable} for an alternative proof) and showed that every continuous representation of $\Gamma_F$ of dimension $n\leq 4$ over $\F_2$ lifts to $\Z/4\Z$. Florence \cite{florence2020smooth} conjectured that the question of Khare and Serre should have a positive answer, and even conjectured the stronger assertion that every finite-dimensional complete flag of continuous representations of $\Gamma_F$ over $k$ should lift to $W_2(k)$. He later constructed, for every odd prime $p$, a $3$-dimensional complete flag of $\Gamma_{\Q(\!(t)\!)}$ which does not lift to $\Z/p^2\Z$, and amended his conjecture to include the assumption that $F$ contains a primitive $p^2$-th root of unity; see \cite{florence2024triangular}.   

    There are also positive results specific to local and global fields. By work of B\"ockle \cite{bockle2003lifting}, all continuous representations $\Gamma_F$-representations over $\F_p$ lift to $\Z/p^2\Z$, when $F$ is a local field. The analogous statement for complete flags has recently been proved by Conti, Demarche and Florence \cite{conti2024lifting}. B\"ockle \cite{bockle2003lifting} also proved lifting of certain mod $p$ representations of $\Gamma_F$, when $F$ is a global field. When $F$ is a number field containing a primitive root of unity of order $p^2$, Khare and Larsen \cite{khare2020liftable} proved that all $3$-dimensional representations of $\Gamma_F$ over $\F_p$ lift to $\Z/p^2\Z$.
    
    \subsection{The main theorem.}  In \cite{merkurjev2024galois}, we showed that the question of Khare and Serre and the conjecture of Florence have a negative answer, even over fields containing all $p$-primary roots of unity. More precisely, for all $n\geq 3$, all odd primes $p$, and all fields $F$ containing a primitive $p$-th root of unity, letting $K\coloneqq F(x_1,\dots,x_p)$, where the $x_i$ are independent variables over $F$, we constructed an $n$-dimensional continuous representation of $\Gamma_K$ with $\F_p$ coefficients, admitting a $\Gamma_K$-invariant complete flag, and which does not lift to $\Z/p^2\Z$. 
    
    After this result, the goal shifted to determining all cases when the question of Khare and Serre has a positive answer, that is, the pairs $(k,n)$, where $k$ is a characteristic $p$ field and $n$ is a positive integer, such that, for every field $K$, every continuous $n$-dimensional representation of $\Gamma_K$ lifts to $W_2(k)$. In this paper, we solve this problem. In fact, we answer a finer, relative version of the problem, where $K$ ranges over all extensions of a fixed field $F$.

\begin{thm}\label{mainthm}
		Let $F$ be a field, let $k$ be a field of characteristic $p>0$, and let $n$ be a positive integer. The following assertions are equivalent.
        \begin{enumerate}
            \item For every field extension $K/F$, every continuous $n$-dimensional representation of $\Gamma_K$ over $k$ lifts to $W_2(k)$.
            \item For every field extension $K/F$, every continuous $n$-dimensional complete flag of representations of $\Gamma_K$ over $k$ lifts to $W_2(k)$.
            \item At least one of the following conditions is satisfied: 
            \begin{enumerate}
                \item $\on{char}(F)=p$,
                \item $n\leq 2$,
                \item $|k|=2$ and $n\leq 4$.
            \end{enumerate}
        \end{enumerate}
	\end{thm}
Thus, for every pair $(k,n)$ except for those considered by Khare and De Clercq--Florence, there exist $n$-dimensional Galois representations over $k$ that do not lift to $W_2(k)$. In fact, in each such case, for every \enquote{generic} extension $K/F$, the Galois group $\Gamma_K$ admits several \enquote{generic} non-liftable representations; see \Cref{explicit} for the precise statement.

\subsection{Negligible cohomology} We rephrase \Cref{mainthm} using the notion of negligible cohomology classes. Consider a short exact sequence of groups
    \begin{equation}\label{general-lift}
\begin{tikzcd}
    0 \arrow[r] & M \arrow[r] & \tilde{G} \arrow[r] & G \arrow[r] & 1,    
\end{tikzcd}
\end{equation}
where $M$ is abelian. The conjugation action of $\tilde{G}$ on $M$ factors through $G$ and makes $M$ into a $G$-module; we let $\alpha\in H^2(G,M)$ be the class of (\ref{general-lift}). Following Serre, we say that $\alpha$ is \emph{negligible over $F$} if for every field extension $K/F$ and every continuous homomorphism $\rho\colon \Gamma_K\to M$ we have $\rho^*(\alpha)=0$ in $H^2(K,M)$; see \Cref{section-negligible} for more details and references. Observe that $\alpha$ is negligible over $F$ if and only if, for all field extensions $K/F$, every continuous homomorphism $\rho\colon \Gamma_K\to G$ lifts to a continuous homomorphism $\tilde{\rho}\colon\Gamma_K\to \tilde{G}$:
\[
    \begin{tikzcd}
    &&& \Gamma_K \arrow[d,"\rho"] \arrow[dl,dotted,swap,"\tilde{\rho}"]  \\
    0 \arrow[r] & M \arrow[r] & \tilde{G} \arrow[r] & G \arrow[r] & 1. 
    \end{tikzcd}
    \]
    
    For every positive integer $n$ and every field $k$ of characteristic $p>0$, we have the short exact sequences of groups
    \begin{equation}
    \begin{tikzcd}
        0\arrow[r] & M_n(k) \arrow[r] & \on{GL}_n(W_2(k)) \arrow[r] & \on{GL}_n(k) \arrow[r] & 1,
    \end{tikzcd}
    \tag*{\text{$\on{GLift}(k,n)$}}
    \end{equation}
     \begin{equation}
    \begin{tikzcd}
        0\arrow[r] & T_n(k) \arrow[r] & B_n(W_2(k)) \arrow[r] & B_n(k) \arrow[r] & 1,
    \end{tikzcd}
    \tag*{\text{$\on{BLift}(k,n)$}}
    \end{equation}
    whose definition is recalled in \Cref{subsection-witt}. Here $B_n\subset \on{GL}_n$ is the Borel subgroup of upper triangular matrices and $T_n(k)\subset M_n(k)$ is the subspace of upper triangular matrices. Since $W_2(\F_p)\cong \Z/p^2\Z$, for $k=\F_p$ these sequences take the form
    \begin{equation}
    \begin{tikzcd}
        0\arrow[r] & M_n(\F_p) \arrow[r] & \on{GL}_n(\Z/p^2\Z) \arrow[r] & \on{GL}_n(\F_p) \arrow[r] & 1,
    \end{tikzcd}
    \tag*{\text{$\on{GLift}(\F_p,n)$}}
    \end{equation}
         \begin{equation}
    \begin{tikzcd}
        0\arrow[r] & T_n(\F_p) \arrow[r] & B_n(\Z/p^2\Z) \arrow[r] & B_n(\F_p) \arrow[r] & 1,
    \end{tikzcd}
    \tag*{\text{$\on{BLift}(\F_p,n)$}}
    \end{equation}
    where the maps $M_n(\F_p)\to \on{GL}_n(\Z/p^2\Z)$ and $T_n(\F_p)\to B_n(\Z/p^2\Z)$ send the matrix $A$ to $I+pA$.

    A continuous $n$-dimensional representation of $\Gamma_K$ over $k$ may be lifted to $W_2(k)$ if and only if the corresponding continuous homomorphism $\Gamma_K\to \on{GL}_n(k)$ (which is uniquely determined up to conjugation) lifts to a continuous homomorphism $\Gamma_K\to \on{GL}_n(W_2(k))$. Similarly an $n$-dimensional complete flag of representations of $\Gamma_K$ over $k$ lifts to $W_2(k)$ if and only if the corresponding continuous homomorphism $\Gamma_K\to B_n(k)$ lifts to $B_n(W_2(k))$.
    Therefore \Cref{mainthm} can be rephrased in the following equivalent way. 

    \begin{center}
\emph{The classes of $\on{GLift}(k,n)$ and $\on{BLift}(k,n)$ are negligible over $F$ if $\on{char}(F)=p$, $n\leq 2$, or $|k|=2$ and $n\leq 4$, and are not negligible over $F$ in all other cases.}
\end{center}

\subsection{Sketch of proof of the main theorem.} Our main tool for the proof of \Cref{mainthm} is \cite[Theorem 1.4]{merkurjev2024galois} (see \Cref{thm-negligible} below). Let $\alpha\in H^2(G,M)$ be the class of (\ref{general-lift}). Suppose that $G$ is a finite group of exponent $e(G)$, that $M$ has finite exponent $e(M)$, and that $F$ contains a primitive root of unity of order $e(M)e(G)$. Under these assumptions, \Cref{thm-negligible} asserts that $\alpha$ is negligible over $F$ if and only if $\alpha$ belongs to the subgroup of $H^2(G,M)$ generated by all elements of the form $\on{cor}^{G_a}_G(a\cup\chi)$, where $a\in M$, $G_a$ is the stabilizer of $a$, and $\chi\in H^2(G_a,\Z)$. When $k$ is finite, \Cref{thm-negligible} reduces \Cref{mainthm} to a problem in finite group cohomology, and when $k$ is infinite, our strategy will be to apply \Cref{thm-negligible} to suitable finite subgroups of $\on{GL}_n(k)$.

We now sketch the proof of \Cref{mainthm}. For clarity, we only consider $\on{GLift}(k,n)$. If $\on{char}(F)=p$, then by \cite[Proposition 3 p. 75]{serre2002galois} the group $H^2(F,M)$ is trivial for every $p$-primary torsion $\Gamma_F$-module $M$, and so \Cref{mainthm} is obvious in this case. When $\on{char}(F)\neq p$, the theorems of Khare and De Clercq--Florence, of which we include self-contained proofs in \Cref{sec:khare-declercq-florence}, deal with all cases when $\on{GLift}(k,n)$ has a positive solution. We must show that $\on{GLift}(k,n)$ is not negligible over $F$ in all the remaining cases. 

For all $n\geq 3$ and all fields $k$ of characteristic $p>0$, if the class of $\on{GLift}(\F_p,n)$ is not negligible over $F$, neither is the class of $\on{GLift}(k,n)$; see \Cref{reduce-to-fp}. Combining this with \Cref{thm-gl3} (whose proof relies on \Cref{thm-negligible}) is enough to conclude when $n\geq 3$ and $p$ is odd. It remains to consider the case when $p=2$ and $n\geq 3$. We consider the cases when $|k|>2$ and $|k|=2$ separately.

The case $p=2$, $|k|>2$ and $n\geq 3$ is handled in \Cref{section-5}. By \Cref{reduce-dimension}, it suffices to consider the case $n=3$. Since $|k|>2$, we may find a Klein subgroup $W$ of $k$. Using $W$, we construct a Klein subgroup $Z\subset \on{GL}_3(k)$ such that (i) the class of $\on{GLift}(k,n)$ does not restrict to zero in $H^2(Z,M_3(k))$ (\Cref{restrict-k-bigger-f2}), while (ii) every class in $H^2(\on{GL}_3(k),M_3(k))$ which is negligible over $F$ is zero in $H^2(Z,M_3(k))$ (\Cref{restriction-kills-negligible}). Statement (i) is proved by a direct matrix computation, while (ii) crucially relies on \Cref{thm-negligible}. Because $\on{GL}_3(k)$ is not necessarily finite, \Cref{thm-negligible} does not apply to $\on{GLift}(k,3)$; we get around this by considering a certain intermediate finite subgroup $Z\subset H\subset \on{GL}_3(k)$ and by proving, using \Cref{thm-negligible}, the stronger statement that all classes in $H^2(H,M_3(k))$ that are negligible over $F$ restrict to zero in $H^2(Z,M_3(k))$. 

We consider the case when $|k|=2$ and $n\geq 5$ in \Cref{section-6}. By \Cref{reduce-dimension}, we may assume that $n=5$. Let $G\coloneqq \on{GL}_5(\F_2)$. By \Cref{thm-negligible}, it suffices to prove that the class of $\on{GLift}(\F_2,5)$ does not belong to the subgroup of $H^2(G,M_5(\F_2))$ generated by the elements $\on{cor}^{G_A}_G(A\cup\chi)$, where $A$ ranges over all elements of $M_5(\F_2)$, $G_A$ is the stabilizer of $A$, and $\chi$ ranges over all elements of $H^2(G_A,\Z)$. It suffices to consider a single $A\in M_5(\F_2)$ for each $G$-orbit. When $A$ is not conjugate to a $5\times 5$ Jordan block, a case-by-case analysis using the projection formula and matrix computations implies that $\on{cor}^{G_A}_G(A\cup\chi)=0$ for all $\chi\in H^2(G_A,\Z)$. When $A$ is conjugate to a $5\times 5$ Jordan block, we have $H^2(G_A,\Z)=(\Z/2\Z)\cdot\chi\oplus(\Z/8\Z)\cdot\psi$ for some $\chi$ and $\psi$. Using the projection formula, we prove that $\on{cor}^{G_A}_G(A\cup\chi)=0$. However, no such argument seems to be available for showing that $\on{cor}^{G_A}_G(A\cup\psi)=0$. To overcome this, in \Cref{restrict-klein-f2} we construct a Klein subgroup $Z\subset \on{GL}_5(\F_2)$ such that (i) the class of $\on{GLift}(\F_2,5)$ restricts to a non-trivial element in $H^2(Z,M_5(\F_2))$ and (ii) the equality $\on{res}^G_Z\on{cor}^{G_A}_G(A\cup \psi)=0$ holds in $H^2(Z,M_5(\F_2))$; see \Cref{lemma-jordan-block}. We prove (i) by a matrix computation, and (ii) by an intricate argument involving the double coset formula. This completes our proof sketch for \Cref{mainthm}.

For completeness, we also determine all cases when $\on{GLift}(k,n)$ is split. In these cases, the corresponding lifting problem is trivial. As we prove in \Cref{lift-split}, the extension $\on{GLift}(k,n)$ is split if and only if either $n=1$, or $n=2$ and $|k|\leq 3$, or $n=3$ and $|k|=2$.

\subsection*{Notation}
For a commutative ring $R$ and a non-negative integer $n$, we let $M_n(R)$ (resp. $T_n(R)$) be the $R$-algebra of $n\times n$ matrices (resp. upper triangular matrices) with coefficients in $R$. We also let $\on{GL}_n(R)$ (resp. $B_n(R)$, resp. $U_n(R)$) be the group of invertible matrices (resp. upper triangular matrices, resp. upper unitriangular matrices) with coefficients in $R$, and we write $R^\times=\on{GL}_1(R)$ for the group of units in $R$. For all $i,j\in\{1,\dots,n\}$, we let $E_{ij}\in M_n(R)$ be the matrix whose $(i,j)$-th entry is equal to $1$ and whose other entries are equal to $0$.

Let $\Gamma$ be a profinite group. All group homomorphisms $\Gamma\to G$, where $G$ is a group, will be assumed to be continuous for the profinite topology on $\Gamma$ and the discrete topology on $G$. All $\Gamma$-modules will be assumed to be discrete. For every $\Gamma$-module $M$ and every non-negative integer $i$, we let $H^i(\Gamma,M)$ be the $i$-th cohomology group.

If $F$ is a field, we let $\Gamma_F$ be the absolute Galois group of $F$ and, for every $\Gamma_F$-module $M$ and every $i\geq 0$, we let $H^i(F,M)\coloneqq H^i(\Gamma_F,M)$.

Let $G$ be a group. For all $\sigma,\tau\in G$, we let $[\sigma,\tau]\coloneqq \sigma\tau\sigma^{-1}\tau^{-1}$ be the commutator of $\sigma$ and $\tau$. We let $[G,G]$ be the derived subgroup of $G$, and we let $G^{\on{ab}}\coloneqq G/[G,G]$ be the abelianization of $G$. 

Let $M$ be a $G$-module. We often view $M$ as a $\Z[G]$-module: for all $\sigma,\tau\in G$ and $m\in M$, we have $(\sigma+\tau)(m)= \sigma (m)+\tau (m)$ and $(\sigma \tau)(m)=\sigma(\tau(m))$. We write $M^G$ for the subgroup of $G$-invariant elements of $M$. For a subgroup $H\subset G$, we let $\on{res}^G_H\colon H^i(G,M)\to H^i(H,M)$ be the restriction map and, if $H$ has finite index in $G$, we let $\on{cor}^H_G\colon H^i(H,M)\to H^i(G,M)$ be the corestriction map. In degree $0$, the corestriction $\on{cor}^H_G\colon H^0(H,M)\to H^0(G,M)$ coincides with the norm map $N_{G/H}\colon M^H\to M^G$; see \cite[p. 48]{neukirch2008cohomology}. For every $\sigma\in G$ and every subgroup $H\subset G$, we let $\sigma_*\colon H^i(H,M)\to H^i(\sigma H\sigma^{-1},M)$ be the conjugation map. By \cite[Proposition 1.5.6]{neukirch2008cohomology}, for any two subgroups $H,K\subset G$ such that $K$ has finite index in $G$, we have the double coset formula 
\[\on{res}^G_H\circ\on{cor}^K_G=\sum_\sigma \on{cor}^{H\cap \sigma K\sigma^{-1}}_H\circ\, \sigma_*\circ\on{res}^K_{K\cap\sigma^{-1} H\sigma},\]
where $\sigma$ ranges over a system of representatives of the double cosets $H\backslash G/K$.

Finally, for every $\sigma\in G$, we write $M^\sigma$ for $M^{\langle \sigma\rangle}$ and $N_\sigma$ for the norm map $N_{\langle \sigma\rangle/\{1\}}\colon M\to M^\sigma$, that is, the map given by $m\mapsto \sum_{i=0}^{e-1}\sigma^im$, where $e$ is the order of $\sigma$.

\section{Preliminaries}\label{section-preliminaries}

\subsection{The lifting problem and negligible cohomology}\label{section-negligible}

Let $G$ be a group, let $M$ be a $G$-module, let $F$ be a field, and let $\alpha\in H^d(G,M)$ be a degree $d$ cohomology class, for some $d\geq 0$. Following Serre, we say that $\alpha$ is \emph{negligible over $F$} if for every field extension $K/F$ and every homomorphism $\Gamma_K\to G$, the pullback map $H^d(G,M)\to H^d(K,M)$ takes $\alpha$ to zero; see \cite{serre1991problemes, serre1994exemples} or \cite[\S 26 p. 61]{garibaldi2003cohomological}. The negligible elements over $F$ form a subgroup
\[
H^2(G,M)_{\on{neg}, F} \subset H^2(G,M).
\]

\begin{lemma}\label{basic-lemma}
	Let $F$ be a field, and let $d$ be a non-negative integer.
	\begin{enumerate}
		\item For every group $G$ and every $G$-module homomorphism $M\to M'$, the induced map
		$H^d(G,M)\to H^d(G,M')$ takes the subgroup $H^d(G,M)_{\on{neg},F}$ into $H^d(G,M')_{\on{neg},F}$.
		\item For every group homomorphism $G'\to G$ and every $G$-module $M$, the pullback map
		$H^d(G,M) \to H^d(G',M)$ takes the subgroup $H^d(G,M)_{\on{neg},F}$ into $H^d(G',M)_{\on{neg},F}$.
        \item For every field extension $F'/F$, every group $G$ and every $G$-module $M$, we have $H^d(G,M)_{\on{neg},F}\subset H^d(G,M)_{\on{neg},F'}$.
        \item For every finite field extension $F'/F$, every group $G$ and every $G$-module $M$, we have $[F':F]\cdot H^d(G,M)_{\on{neg},F'}\subset H^d(G,M)_{\on{neg},F}$.
	\end{enumerate}
\end{lemma}

\begin{proof}
    The proofs immediately follow from the definitions; see \cite[Proposition 2.3]{gherman2022negligible}, where the assumption that $G$ is finite is unnecessary.
\end{proof}

Let $G$ be a group, let $M$ be a $G$-module, and let $F$ be a field. Consider a group extension 
\begin{equation}\label{extension-eq}
	\begin{tikzcd}
		0 \arrow[r]  & M \arrow[r] & \tilde{G} \arrow[r] & G \arrow[r] & 1,
	\end{tikzcd}
\end{equation}
where the $G$-action on $M$ induced by the conjugation $\tilde{G}$-action coincides with the $G$-module action, and let $\alpha\in H^2(G,M)$ be the class of (\ref{extension-eq}). The class $\alpha$ is negligible if and only if, for every field extension $K/F$, every homomorphism $\Gamma_K\to G$ lifts to a homomorphism $\Gamma_K\to \tilde{G}$. 

In \cite{merkurjev2024galois}, we determined the subgroup $H^2(G,M)_{\on{neg},F}\subset H^2(G,M)$ when $G$ is finite, $M$ has finite exponent, and $F$ contains enough roots of unity.

\begin{thm}\label{thm-negligible}
	Let $G$ be a finite group of exponent $e(G)$, let $M$ be a $G$-module of finite exponent $e(M)$, and let $F$ be a field containing a primitive root of unity of order $e(M) e(G)$. Then $H^2(G,M)_{\on{neg},F}$ is generated by all elements of the form $\on{cor}^H_G(a\cup\chi)$, where $H$ is a subgroup of $G$, $a\in M^H$ and $\chi\in H^2(H,\Z)$.
    
    In fact, $H^2(G,M)_{\on{neg},F}$ is generated by all elements of the form $\on{cor}^{G_a}_G(a\cup\chi)$, where $a$ ranges over all elements of $M$, $G_a$ is the stabilizer of $a$ in $G$, and $\chi$ ranges over all elements of $H^2(G_a,\Z)$.
\end{thm}

\begin{proof}
    When $M$ is finite, this is \cite[Theorem 1.3]{merkurjev2024galois}. The general case follows from the finite case by writing $M$ as the union of its finite $G$-submodules.
\end{proof}

Suppose that the group $G$ is finite. For every subgroup $H\subset G$, we define
	\begin{equation}\label{varphi}\varphi_{H}\colon M^H\otimes H^2(H,\Z)\xrightarrow{\cup} H^2(H,M)\xrightarrow{\on{cor}} H^2(G, M).
    \end{equation}
Therefore, under the assumptions of \Cref{thm-negligible}, the subgroup $H^2(G,M)_{\on{neg},F}$ is generated by the images of $\varphi_{G_a}$, where $a$ ranges over all elements of $M$, and where $G_a$ is the stabilizer of $a$ in $G$. In fact, as the next lemma shows, it suffices to consider a single $a\in M$ for each $G$-orbit.

\begin{lemma}\label{modulation}
    Let $G$ be a finite group, let $M$ be a $G$-module. For every $g\in G$ and every $a\in M$, we have $\on{Im}(\varphi_{ga})=\on{Im}(\varphi_a)$.
\end{lemma}

\begin{proof}
Let $a\in M$, let $g\in G$, and set $a'\coloneqq ga$. We have a commutative diagram
\[
\begin{tikzcd}
    M^{G_a}\otimes H^2(G_a,\Z) \arrow[r,"\cup"] \arrow[d,"g_*\otimes g_*"] & H^2(G_a,M) \arrow[d,"g_*"] \arrow[r,"\on{cor}"] & H^2(G,M) \arrow[d,"g_*"]  \\
    M^{G_{a'}}\otimes H^2(G_{a'},\Z) \arrow[r,"\cup"] & H^2(G_{a'},M) \arrow[r,"\on{cor}"] & H^2(G,M).
\end{tikzcd}
\]
Here, the left square commutes by \cite[Proposition 1.5.3(i)]{neukirch2008cohomology}, and the right square commutes by \cite[Proposition 1.5.4]{neukirch2008cohomology}. Moreover, by \cite[Theorem 6.7.8]{weibel1994introduction} the right vertical map is the identity. It follows that $\on{Im}(\varphi_{G_a})=\on{Im}(\varphi_{G_{a'}})$.
\end{proof}

\subsection{Length \texorpdfstring{$2$}{2} Witt vectors}\label{subsection-witt}

Let $k$ be a field of characteristic $p>0$. We recall the definition of the $p$-typical length $2$ Witt vectors $W_2(k)$ of $k$. Consider the polynomial $\Phi(x,y)\coloneqq ((x+y)^p-x^p-y^p)/p\in \Z[x,y]$. As a set $W_2(k)\coloneqq k\times k$, and, for all $(a_1,b_1)$, $(a_2,b_2)\in W_2(k)$, one has 
\[(a_1, b_1) + (a_2, b_2) \coloneqq (a_1 + a_2, b_1 + b_2 - \Phi(a_1,a_2)),\]
\[(a_1,b_1)\cdot (a_2,b_2)\coloneqq (a_1a_2, a_1^pb_2+a_2^pb_1);\]
see \cite[Chapitre IX, \S 1, paragraphe 4]{bourbaki2006algebre}.
We have a short exact sequence of abelian groups
\begin{equation}\label{length-2-witt-eq}\begin{tikzcd}
    0 \arrow[r] & k \arrow[r,"\iota"] & W_2(k) \arrow[r,"\pi"] & k \arrow[r] & 0,  
\end{tikzcd}
\end{equation}
where $\pi(a,b)=a$ and $\iota(b)=(0,b)$ for all $a,b\in k$. The map $\pi$ is a ring homomorphism. For every integer $n\geq 0$, we obtain a short exact sequence of groups
\begin{equation}\label{lifting-sequence-section-2}
    \begin{tikzcd}
        0\arrow[r] & M_n(k) \arrow[r] & \on{GL}_n(W_2(k)) \arrow[r] & \on{GL}_n(k) \arrow[r] & 1,
    \end{tikzcd}
    \tag*{\text{Lift}($k$,$n$)}
    \end{equation}
where the homomorphism $\on{GL}_n(W_2(k))\to \on{GL}_n(k)$ is induced by $\pi$, and where the inclusion $M_n(k)\to \on{GL}_n(W_2(k))$ is given by $(m_{ij})\mapsto I+(\iota(m_{ij}))$. Similarly, we have an exact sequence 
     \begin{equation}
    \begin{tikzcd}
        0\arrow[r] & T_n(k) \arrow[r] & B_n(W_2(k)) \arrow[r] & B_n(k) \arrow[r] & 1.
    \end{tikzcd}
    \tag*{\text{$\on{BLift}(k,n)$}}
    \end{equation}
For every $A=(a_{ij})\in \on{GL}_n(k)$, we define $A^{(p)}\coloneqq (a^p_{ij})\in \on{GL}_n(k)$.

\begin{lemma}\label{induced-action}
    The $\on{GL}_n(k)$-action on $M_n(k)$ induced by $\on{GLift}(k,n)$ is given by
\[\on{GL}_n(k)\times M_n(k)\to M_n(k),\qquad (A,M)\mapsto A^{(p)}M(A^{(p)})^{-1}.\]
\end{lemma}
\begin{proof}
    Under the identification $M_n(W_2(k)) = M_n(k) \times M_n(k)$ induced by the identification $W_2(k)=k\times k$, the conclusion amounts to \[(A,0)(0,M)(A,0)^{-1}=(0,A^{(p)}M(A^{(p)})^{-1})\] for all $A\in \on{GL}_n(k)$ and $M\in M_n(k)$. This is equivalent to
    \[(A,0)(0,M)=(0,A^{(p)}M(A^{(p)})^{-1})(A,0),\]
    which follows from the identities \[(X,0) (0,Y) = (0, X^{(p)} Y),\qquad (0,Y) (X,0) = (0, Y X^{(p)}),\]
    valid for all $X,Y\in M_n(k)$.
\end{proof}

When $k=\F_p$, we have a ring isomorphism $\Z/p^2\Z \xrightarrow{\sim} W_2(\F_p)$ determined by $1+p^2\Z \mapsto (1,0)$. Thus (\ref{length-2-witt-eq}) becomes
\[
\begin{tikzcd}
0\arrow[r] &  \F_p \arrow[r] & \Z/p^2\Z \arrow[r] &  \F_p \arrow[r] &  0,
\end{tikzcd}\]
where the map $\F_p\to \Z/p^2\Z$ sends $1$ to $p+p^2\Z$, and the sequences $\on{GLift}(\F_p,n)$ and $\on{BLift}(\F_p,n)$ take the form
\begin{equation}
\begin{tikzcd}
0\arrow[r] & M_n(\F_p) \arrow[r] & \on{GL}_n(\Z/p^2\Z) \arrow[r] & \on{GL}_n(\F_p) \arrow[r] & 1,
\end{tikzcd}
    \tag*{\text{$\on{GLift}(\F_p,n)$}}
    \end{equation}
\begin{equation}
    \begin{tikzcd}
        0\arrow[r] & T_n(\F_p) \arrow[r] & B_n(\Z/p^2\Z) \arrow[r] & B_n(\F_p) \arrow[r] & 1,
    \end{tikzcd}
    \tag*{\text{$\on{BLift}(\F_p,n)$}}
    \end{equation}
where the maps $M_n(\F_p) \to \on{GL}_n(\Z/p^2\Z)$ and $T_n(\F_p)\to B_n(\Z/p^2\Z)$ send $A$ to $I+pA$. By \Cref{induced-action}, the induced $\on{GL}_n(\F_p)$-action on $M_n(\F_p)$ is given by matrix conjugation.

\begin{thm}\label{thm-gl3}
	For all $n\geq 3$, all odd primes $p$, and all fields $F$ of characteristic different from $p$, the class of $\on{GLift}(\F_p,n)$ is not negligible over $F$.
\end{thm}   

\begin{proof}
    See \cite[Theorem 5.1]{merkurjev2024galois}.
\end{proof}

We conclude this subsection with some basic observations about $\on{GLift}(k,n)$ and $\on{BLift}(k,n)$.

\begin{lemma}\label{prime-to-p-infinite-k}
    Let $\Gamma$ be a profinite group, let $k$ be a field of characteristic $p>0$, and let $V$ be a finite-dimensional $k$-representation of $\Gamma$. There exists an open subgroup $\Gamma'\subset \Gamma$ of prime-to-$p$ index such that $V$ is a unitriangular representation of $\Gamma'$.
\end{lemma}

\begin{proof}
   Replacing $\Gamma$ with the image of the natural homomorphism $\Gamma\to \on{Aut}(V)$, we may assume that $\Gamma$ is finite. Replacing $\Gamma$ by a $p$-Sylow subgroup, we may assume that $\Gamma$ is a $p$-group. By induction on the dimension of $V$, it suffices to show that $V^\Gamma\neq \{0\}$. This is proved in \cite[Proposition 26 p. 64]{serre2012linear}.
\end{proof}

\begin{lemma}\label{blift-implies-glift}
    Let $F$ be a field, let $k$ be a field of characteristic $p>0$, and let $n$ be a positive integer. If the class of $\on{BLift}(k,n)$ is negligible over $F$, then so is the class of $\on{GLift}(k,n)$.
\end{lemma}

\begin{proof}
    Let $K/F$ be a field extension, and let $\rho\colon \Gamma_K\to \on{GL}_n(k)$ be a group homomorphism. Let $G\subset \on{GL}_n(k)$ be the image of $\rho$, and let $P\subset G$ be a $p$-Sylow subgroup of $G$. Since $P$ is a finite $p$-group, by \Cref{prime-to-p-infinite-k}, we may assume that $P \subset B_n(k)$. Let $c\in H^2(B_n(k), M_n(k))$ be the pushforward of the class of $\on{BLift}(k,n)$. By \Cref{basic-lemma}(1), the class $c$ is negligible over $F$, and hence so is its restriction in $H^2(P, M_n(k))$. The latter class is the restriction of the class of $\on{GLift}(k,n)$ via the inclusion $P\hookrightarrow G \hookrightarrow \on{GL}_n(k)$.
As $[G:P]$ is prime to $p$, by \Cref{basic-lemma}(4) the restriction in $H^2(G, M_n(k))$ of the class of $\on{GLift}(k,n)$ is also negligible over $F$. It follows that $\rho$ lifts to $\on{GL}_n(W_2(k))$.
\end{proof}

\begin{lemma}\label{reduce-dimension}
    Let $F$ be a field, let $k$ be a field of characteristic $p>0$, and let $n\geq m$ be positive integers. If the class of $\on{GLift}(k,n)$ is negligible over $F$, then so is the class of $\on{GLift}(k,m)$. Similarly, if the class of $\on{BLift}(k,n)$ is negligible over $F$, then so is the class of $\on{BLift}(k,m)$.
\end{lemma}

\begin{proof}
See \cite[Lemma 3.4]{declercq2022lifting}. For a more direct argument, see \cite[Lemma 5.3]{merkurjev2024galois}, which is stated and proved only when $k=\F_p$, but whose proof immediately generalizes to arbitrary $k$.
\end{proof}

\subsection{Extensions of bicyclic groups}

    Let $s$ and $t$ be positive integers, let \[Z\coloneqq \ang{\rho,\mu\, |\,  \rho^s=\mu^t=[\rho,\mu]=1}\] be a bicyclic group of order $st$, and let $M$ be a $Z$-module. Define the abelian group \[
	Z^2(Z,M)\coloneqq \{(a, b, c) \in M^3 \mid \rho(a) = a, \mu(b) = b, N_\rho(c) = (\mu-1)a, N_\mu(c) = (\rho-1)b\},
	\]
	its subgroup
	\[
	B^2(Z,M)=  \{(N_\rho(u), N_\mu(v), (\rho - 1)v + (\mu - 1)u)\, |\,  u, v \in M\},
	\]
    and the quotient group
    \[\tilde{H}^2(Z,M)\coloneqq Z^2(Z,M)/B^2(Z,M).\]
    Let $\alpha\in H^2(Z,M)$, and let
    \begin{equation}\label{m-tilde-z}
    \begin{tikzcd}
        0 \arrow[r] & M \arrow[r] & \tilde{Z} \arrow[r] & Z \arrow[r] & 1.   
    \end{tikzcd}
    \end{equation}
    be a group extension representing $\alpha$. Let $\tilde{\rho},\tilde{\mu}\in \tilde{Z}$ be lifts of $\rho$ and $\mu$, respectively. Observe that $\tilde{\rho}^{-s}$, $\tilde{\mu}^t$ and $[\tilde{\rho},\tilde{\mu}]$ belong to $M$, and the triple $(\tilde{\rho}^{-s},\tilde{\mu}^t,[\tilde{\rho},\tilde{\mu}])$ belongs to $Z^2(Z,M)$. Define 
    \[f(\alpha)\coloneqq (\tilde{\rho}^{-s},\tilde{\mu}^t,[\tilde{\rho},\tilde{\mu}])+B^2(Z,M) \in \tilde{H}^2(Z,M).\]

\begin{lemma}\label{klein-cohomology}
	This construction yields a well-defined function \[f\colon H^2(Z,M)\to \tilde{H}^2(Z,M)\]
    such that $f(0)=0$.
\end{lemma}

\begin{proof}
	Let $\alpha\in H^2(Z,M)$, let (\ref{m-tilde-z}) be a group extension representing $\alpha$, and let $\tilde{\rho},\tilde{\mu}\in \tilde{Z}$ be lifts of $\rho$ and $\mu$, respectively.
    
    We first show that $f(\alpha)$ does not depend on the choice of lifts $\tilde{\rho},\tilde{\mu}$. Any other pair of lifts has the form $u^{-1}\tilde{\rho},v\tilde{\mu}$ for some $u,v\in M$. (Here we view $M$ as a subgroup of $\tilde{Z}$, and hence use multiplicative notation for the group operation.) Then
    \[(u^{-1}\tilde{\rho})^{-s}=\tilde{\rho}^{-s}N_\rho(u),\qquad (v\tilde{\mu})^t=N_\mu(v)\tilde{\mu}^t,\]
    \[[u^{-1}\tilde{\rho},v\tilde{\mu}]=u^{-1}\tilde{\rho}v\tilde{\rho}^{-1}\tilde{\rho}\tilde{\mu}\tilde{\rho}^{-1}\tilde{\mu}^{-1}\tilde{\mu}u\tilde{\mu}^{-1}v^{-1}=u^{-1}\rho(v)[\tilde{\rho},\tilde{\mu}]\mu(u)v^{-1}.\]
    Recalling that the subgroup $M$ is abelian, we obtain, in additive notation, 
    \[((u^{-1}\tilde{\rho})^{-s},(v\tilde{\mu})^t,[u^{-1}\tilde{\rho},v\tilde{\mu}])= (\tilde{\rho}^{-s},\tilde{\mu}^t,[\tilde{\rho},\tilde{\mu}]) + (N_\rho(u), N_\mu(v), (\rho - 1)v + (\mu - 1)u)\]
    in $Z^2(Z,M)$. Thus $f(\alpha)$ does not depend on the choice of lift. The fact that $f(\alpha)$ does not depend on the choice of the group extension (\ref{m-tilde-z}) is clear. Finally, if $\alpha=0$ then (\ref{m-tilde-z}) admits a splitting $s\colon Z\to \tilde{Z}$. Letting $\tilde{\rho}\coloneqq s(\rho)$ and $\tilde{\mu}\coloneqq s(\mu)$, we have $(\tilde{\rho}^{-s},\tilde{\mu}^t,[\tilde{\rho},\tilde{\mu}])=0$ in $Z^2(Z,M)$, and hence $f(0)=0$.
\end{proof}

\begin{rmk}
    One can show that the function $f\colon H^2(Z,M)\to \tilde{H}^2(Z,M)$ is a group isomorphism. We will not need this stronger assertion.
\end{rmk}

\section{Proofs of the theorems of Khare and De Clercq--Florence}\label{sec:khare-declercq-florence}

Let $R$ be a commutative ring, let $\Gamma$ be a profinite group, and let $A$ and $C$ be $R[\Gamma]$-modules. We let 
\[\on{Ext}^1_{R[\Gamma],s}(C,A)\coloneqq \on{Ker}[\on{Ext}^1_{R[\Gamma]}(C,A)\to \on{Ext}^1_{R}(C,A)]\]
be the abelian group of isomorphism classes of $R$-split exact sequences of $R[\Gamma]$-modules
\[
\begin{tikzcd}
    0\arrow[r] & A \arrow[r] & B \arrow[r] & C \arrow[r] & 0.  
\end{tikzcd}
\]

Given a (continuous) $1$-cocycle $\varphi$ of $\Gamma$ with values in $\operatorname{Hom}_R(C,A)$, one introduces the
structure of an $R[\Gamma]$-module on $A \oplus C$ by the formula
\[
g(a,c) = (ga + \varphi(g)(gc), gc).
\]
This yields a group isomorphism 
\begin{equation}\label{h1-extensions}H^1(\Gamma, \operatorname{Hom}_R(C,A))\xrightarrow{\sim} \on{Ext}^1_{R[\Gamma],s}(C,A).\end{equation}
For every ring homomorphism $R\to R'$, letting $A'\coloneqq A\otimes_RR'$ and $C'\coloneqq C\otimes_RR'$, base change induces a commutative square
\begin{equation}\label{h1-extensions-base-change}
\begin{tikzcd}
    H^1(\Gamma, \operatorname{Hom}_R(C,A))\arrow[r,"\sim"] \arrow[d] & \on{Ext}^1_{R[\Gamma],s}(C,A) \arrow[d]  \\
    H^1(\Gamma, \operatorname{Hom}_{R'}(C',A'))\arrow[r,"\sim"] & \on{Ext}^1_{R'[\Gamma],s}(C',A'),
\end{tikzcd}
\end{equation}
where the bottom horizontal map is (\ref{h1-extensions}) for the $R'$-modules $A'$ and $C'$.

The following theorem was proved by Khare \cite{khare1997base} when the field $k$ is finite, and by De Clercq and Florence \cite{declercq2022lifting} in general.

\begin{thm}\label{kdf}
  For every field $F$ and every field $k$ of characteristic $p>0$, the classes of $\on{GLift}(k,2)$ and $\on{BLift}(k,2)$ are negligible over $F$.
\end{thm}

\begin{proof}
By \Cref{blift-implies-glift}, it suffices to prove that the class of $\on{BLift}(k,2)$ is negligible over $F$. By \cite[Proposition 3 p. 75]{serre2002galois}, we may assume that $\on{char}(F)\neq p$ and, by \Cref{basic-lemma}(4), that $F$ contains a primitive $p$-th root of unity $\zeta$. The choice of $\zeta$ allows us to identify $\mu_p$ with $\mathbb{Z}/p\mathbb{Z}$ and $k\otimes\mu_p$ with $k$. Let
\begin{equation}\label{gl2-proof-seq}
  \begin{tikzcd}
      0 \arrow[r] & k \arrow[r] & V \arrow[r] & k \arrow[r] & 0    
  \end{tikzcd}
  \end{equation}
be a $2$-dimensional complete flag of representations of $\Gamma_F$ over $k$. By \Cref{prime-to-p-infinite-k}, there exists an open subgroup $\Gamma'\subset \Gamma$ acting trivially on both copies of $k$ in (\ref{gl2-proof-seq}). By \Cref{basic-lemma}(4), we may replace $\Gamma$ by $\Gamma'$, that is, we may assume that $\Gamma_F$ acts trivially on both copies of $k$ in (\ref{gl2-proof-seq}). Consider the commutative diagram
  \[
  \begin{tikzcd}
      W_2(k)\otimes F^\times \arrow[r,"\sim"] \arrow[d] & W_2(k)\otimes H^1(F,\mu_{p^2}) \arrow[d]  \arrow[r,"\cup"] \arrow[d] & H^1(F,W_2(k)\otimes \mu_{p^2}) \arrow[d] \\
      k\otimes F^\times \arrow[r,"\sim"] & k\otimes H^1(F,\mu_p) \arrow[r,"\cup"] & H^1(F,k), 
  \end{tikzcd}
  \]
    where the $\Gamma_F$-action on $k$ and $W_2(k)$ is trivial,  the vertical maps are induced by the reduction map $W_2(k)\to k$, and the left horizontal maps are induced by the Kummer sequence. As the map $W_2(k)\to k$ is surjective, so is $W_2(k)\otimes F^\times\to k\otimes F^\times$. Since $k$ is an $\F_p$-vector space, the bottom-right map is an isomorphism, and hence the homomorphism $H^1(F,W_2(k)\otimes \mu_{p^2})\to H^1(F,k)$ is also surjective.
    
    In view of (\ref{h1-extensions}), the extension (\ref{gl2-proof-seq}) is represented by a class $\alpha\in H^1(F,k)$ and, letting $\tilde{\alpha}\in H^1(F,W_2(k)\otimes \mu_{p^2})$ be a lift of $\alpha$, the class $\tilde{\alpha}$ represents a $W_2(k)$-split extension of $W_2(k)[\Gamma_F]$-modules 
  \begin{equation}\label{gl2-proof-seq-2}
  \begin{tikzcd}
      0 \arrow[r] &  W_2(k) \otimes \mu_{p^2} \arrow[r] & W \arrow[r] & W_2(k) \arrow[r] & 0.
  \end{tikzcd}
  \end{equation}
  Since $\tilde{\alpha}$ lifts $\alpha$, the commutativity of (\ref{h1-extensions-base-change}) (where the homomorphism $R\to R'$ is the reduction map $W_2(k)\to k$) implies that tensoring (\ref{gl2-proof-seq-2}) with $k$ over $W_2(k)$ yields (\ref{gl2-proof-seq}), and the conclusion follows.
\end{proof}

The next theorem is due to De Clercq and Florence \cite[Corollary 6.3]{declercq2022lifting}.
\begin{thm}\label{df}
  For every field $F$ and every $n\leq 4$, the classes of $\on{GLift}(\F_2,n)$ and $\on{BLift}(\F_2,n)$ are negligible over $F$.
\end{thm}

\begin{proof}
By \Cref{blift-implies-glift}, it suffices to prove that $\on{BLift}(\F_2,n)$ is negligible over $F$. By \cite[Proposition 3 p. 75]{serre2002galois}, we may assume that $\on{char}(F)\neq 2$ and, by \Cref{reduce-dimension}, that $n = 4$. 

Let $V_1\subset V_2\subset V_3\subset V_4=V$ be a $4$-dimensional complete flag of representations of $\Gamma_F$ over $\F_2$. Every triangular action of $\Gamma_F$ on a $2$-dimensional vector space over $\F_2$ has a permutation basis: this is clear if the $\Gamma_F$-action is trivial, and if the $\Gamma_F$-action is non-trivial, then the $\Gamma_F$-orbit of a non-fixed vector is a permutation basis. Thus, there exist $\Gamma_F$-invariant bases $X=\{x_1,x_2\}$ of $V_2$ and $Y=\{y_1,y_2\}$ of $V/V_2$ such that
\[V_2=\F_2[X],\quad V/V_2=\F_2[Y],\quad V_1=\F_2\cdot(x_1+x_2),\quad V_3/V_2=\F_2\cdot (y_1+y_2).\]
We obtain a short exact sequence of $\F_2$-linear $\Gamma_F$-representations
\begin{equation}\label{x-v-y}
\begin{tikzcd}
    0 \arrow[r] & \F_2[X] \arrow[r] & V \arrow[r] & \F_2[Y] \arrow[r] & 0.   
\end{tikzcd}
\end{equation}

  Let $L$ be an \'{e}tale $F$-algebra corresponding to the $\Gamma_F$-set $X \times Y$; see \cite[Theorem 18.4]{knus1998book}. We have a commutative square of $(\Z/4\Z)[\Gamma_F]$-modules
  \[
  \begin{tikzcd}
       \on{Hom}_{\Z/4\Z}((\Z/4\Z)[Y], \mu_4[X]) \arrow[r,"\sim"] \arrow[d] & \mu_4[X \times Y] \arrow[d] \\ 
       \on{Hom}_{\F_2}(\F_2[Y], \F_2[X]) \arrow[r,"\sim"] & \F_2[X \times Y],
  \end{tikzcd}
  \]
  where $\mu_4[X]\coloneqq (\Z/4\Z)[X]\otimes \mu_4$ and $\mu_4[X \times Y]\coloneqq (\Z/4\Z)[X\times Y]\otimes \mu_4$. We obtain a commutative diagram
  \[
  \begin{tikzcd}
  H^1(L,\mu_4) \arrow[d] \arrow[r,"\sim"] & H^1(F,\mu_4[X\times Y]) \arrow[r,"\sim"] \arrow[d] & H^1(F,\on{Hom}_{\Z/4\Z}((\Z/4\Z)[Y], \mu_4[X])) \arrow[d] \\ 
  H^1(L,\F_2) \arrow[r,"\sim"] & H^1(F,\F_2[X\times Y]) \arrow[r,"\sim"] & H^1(F,\on{Hom}_{\F_2}(\F_2[Y], \F_2[X])),
  \end{tikzcd}
  \]
  where the three vertical arrows are induced by the reduction maps $\Z/4\Z\to \F_2$ and $\mu_4\to \F_2$, and where the left horizontal maps are the Faddeev--Shapiro isomorphisms; see \cite[Proposition 1.6.4]{neukirch2008cohomology}. The map $H^1(L,\mu_4)\to H^1(L,\F_2)$ is surjective by Kummer theory, and hence all vertical maps are surjective.

  Let $\alpha\in H^1(F,\on{Hom}_{\F_2}(\F_2[Y], \F_2[X]))$ be the class of (\ref{x-v-y}), and lift $\alpha$ to an element $\tilde{\alpha}\in H^1(F,\on{Hom}_{\Z/4\Z}((\Z/4\Z)[Y], \mu_4[X]))$. Then, under the identification of (\ref{h1-extensions}), $\tilde{\alpha}$ is the class of a $(\Z/4\Z)$-split exact sequence of $(\mathbb{Z}/4\mathbb{Z})[\Gamma_F]$-modules 
  \[
  \begin{tikzcd}
       0 \arrow[r] & \mu_4[X] \arrow[r] & W \arrow[r] & (\mathbb{Z}/4\mathbb{Z})[Y] \arrow[r] & 0
  \end{tikzcd}
  \]
  which reduces to (\ref{x-v-y}) modulo $2$. Define $W_1\coloneqq \mu_4\cdot(x_1+x_2)$, $W_2\coloneqq \mu_4[X]$, let $W_3$ be the inverse image of $\Z/4\Z\cdot (y_1+y_2)$ in $W$, and let $W_4\coloneqq W$. Then the flag of $\Z/4\Z$-free $\Gamma_F$-modules $W_1\subset W_2\subset W_3\subset W_4$ reduces to the flag $V_1\subset V_2\subset V_3\subset V_4$ modulo $2$, as desired.
\end{proof}

\section{Proof of Theorem \ref{mainthm} for odd \texorpdfstring{$p$}{p} and \texorpdfstring{$n\geq 3$}{n≥3}}\label{sec:p-odd}

\begin{lemma}\label{reduce-to-fp}
	Let $p$ be a prime number, let $n$ be a positive integer, let $F$ be a field, and let $k$ be a field of characteristic $p$. If $\on{GLift}(\F_p,n)$ is not negligible over $F$, neither is $\on{GLift}(k,n)$. Similarly, if $\on{BLift}(\F_p,n)$ is not negligible over $F$, neither is $\on{BLift}(k,n)$.
\end{lemma}

\begin{proof}
	We have a commutative diagram of abelian groups with exact rows
	\[
	\begin{tikzcd}
		0 \arrow[r]  &  k\arrow[r] & W_2(k)\arrow[r] & k\arrow[r] & 0 \\
		0 \arrow[r] & k \arrow[r]\arrow[d,"\lambda"]\arrow[u,equal]   & C\arrow[r]\arrow[u,hook] \arrow[d,"\varphi"]  & \F_p\arrow[r]\arrow[u,hook] \arrow[d,equal] & 0 \\
		0 \arrow[r] & \F_p \arrow[r] & W_2(\F_p)\arrow[r] & \F_p\arrow[r] & 0,
	\end{tikzcd}
	\]
	where the group homomorphism $\lambda$ is a splitting of the inclusion $\F_p\hookrightarrow k$. By definition, $C$ is the subring of $W_2(k)$ consisting of those pairs $(a,b)$ such that $a\in \F_p$ and $b\in k$. The ring homomorphism $\varphi$ is given by $\varphi(a,b)=(a,\lambda(b))$.
	
	We obtain a commutative diagram of groups with exact rows
	\[
	\begin{tikzcd}
		0 \arrow[r]  & M_n(k)\arrow[r] & \on{GL}_n(W_2(k))\arrow[r] & \on{GL}_n(k)\arrow[r] & 1 \\
		0 \arrow[r] & M_n(k) \arrow[r]\arrow[d,"\lambda_*"]\arrow[u,equal]   & \on{GL}_n(C)\arrow[r]\arrow[u,hook] \arrow[d,"\varphi_*"]  & \on{GL}_n(\F_p)\arrow[r]\arrow[u,hook] \arrow[d,equal] & 1 \\
		0 \arrow[r] & M_n(\F_p) \arrow[r] & \on{GL}_n(W_2(\F_p))\arrow[r] & \on{GL}_n(\F_p)\arrow[r] & 1.
	\end{tikzcd}
	\]
	Since the top row is negligible over $F$, by \Cref{basic-lemma}(2) so is the middle row, and hence by \Cref{basic-lemma}(1) so is the bottom row. The proof for $\on{BLift}(k,n)$ is entirely analogous.
\end{proof}

\begin{proof}[Proof of \Cref{mainthm} for odd $p$ and $n\geq 3$]
	By \Cref{blift-implies-glift}, it suffices to prove that $\on{GLift}(k,n)$ is not negligible over $F$ for all $n\geq 3$. By \Cref{reduce-to-fp}, it is enough to show that the class of $\on{GLift}(\F_p,n)$ is not negligible over $F$, which follows from \Cref{thm-gl3}.
\end{proof}

\section{Proof of Theorem \ref{mainthm} for \texorpdfstring{$p=2$}{p=2}, \texorpdfstring{$|k|>2$}{|k|>2} and \texorpdfstring{$n\geq 3$}{n≥3}}\label{section-5}

\begin{lemma}\label{restrict-k-bigger-f2}
	Let $n\geq 2$ be an integer, let $k$ be a field of characteristic $2$ such that $|k|>2$, let $x,y\in k^\times$ be two distinct elements, and let $Z\subset U_n(k)$ be the Klein subgroup generated by $\rho\coloneqq I+xE_{1,n}$ and $\mu \coloneqq I+yE_{1,n}$. The class of $\on{GLift}(k,n)$	in $H^2(\on{GL}_n(k), M_n(k))$ restricts to a non-trivial class in $H^2(Z,M_n(k))$.
\end{lemma}

\begin{proof}
    We first reduce to the case when $n=2$. We have a commutative diagram with exact rows
    \[
    \begin{tikzcd}
        0 \arrow[r] & M_n(k) \arrow[r]  & \on{GL}_n(W_2(k)) \arrow[r] & \on{GL}_n(k) \arrow[r] & 1 \\
        0 \arrow[r] & M_n(k) \arrow[u,equal]\arrow[r]\arrow[d,->>,"\pi"]  & E \arrow[r] \arrow[u]\arrow[d] & \on{GL}_2(k) \arrow[r] \arrow[u,hook,"\iota"] \arrow[d,equal]  & 1 \\
        0 \arrow[r] & M_2(k) \arrow[r] & \on{GL}_2(W_2(k)) \arrow[r] & \on{GL}_2(k) \arrow[r] & 1,
    \end{tikzcd}
    \]
    where $\pi$ and $\iota$ are given by
    \[[a_{ij}]\mapsto \begin{bmatrix}
        a_{1,1} & a_{1,n} \\
        a_{n,1} & a_{n,n}
    \end{bmatrix},\qquad \begin{pmatrix}
        a & b \\
        c & d
    \end{pmatrix}\mapsto
\begin{pmatrix}
a & 0 & \cdots & 0 & b \\
0 & 1 & \cdots & 0 & 0 \\
\vdots & \vdots & \ddots & \vdots & \vdots \\
0 & 0 & \cdots & 1 & 0 \\
c & 0 & \cdots & 0 & d
\end{pmatrix},
\]
respectively. Letting $\alpha_n\in H^2(\on{GL}_n(k),M_n(k))$ be the class of $\on{GLift}(k,n)$, we deduce that $\alpha_2=\pi_*\iota^*(\alpha_n)$. Let $j_n\colon Z\hookrightarrow \on{GL}_n(k)$ be the inclusion map. We have $\iota\circ j_2=j_n$, so that $j_n^*=j_2^*\iota^*$, and we have $j_2^*\pi_*=\pi_*j_2^*$. Thus
\[j_2^*(\alpha_2)=j_2^*\pi_*\iota^*(\alpha_n)=\pi_*j_2^*\iota^*(\alpha_n)=\pi_*j_n^*(\alpha_n)\]
in $H^2(Z,M_2(k))$.
In particular, if $j_2^*(\alpha_2)\neq 0$ in $H^2(Z,M_2(k))$, then $j_n^*(\alpha_n)\neq 0$ in $H^2(Z,M_n(k))$. We may thus assume that $n=2$.

	Let $\tilde{x}\coloneqq (x,0)$ and $\tilde{y}\coloneqq (y,0)$ be lifts of $x$ and $y$ in $W_2(k)$, respectively, and define $\tilde{\rho} \coloneqq I+\tilde{x}E_{12}$ and $\tilde{\mu}\coloneqq I+\tilde{y}E_{12}$ in $\mathrm{GL}_2(W_2(k))$. Then $\tilde{\rho}$ and $\tilde{\mu}$ lift $\rho$ and $\mu$, respectively. For every $u\in k$, we have $(u,0)+(u,0)=(0,u^2)=\iota(u^2)$ in $W_2(k)$, where the map $\iota\colon k\to W_2(k)$ has been defined in (\ref{length-2-witt-eq}). Thus
    \[\tilde{\rho}^{-2} =I - \iota(x^2)E_{12}= I+\iota(x^2)E_{12},\qquad \tilde{\mu}^2= I + \iota(y^2)E_{12},\qquad [\tilde{\rho},\tilde{\mu}]=I\]
    in $\on{GL}_2(W_2(k))$. 
    
    Suppose by contradiction that $\on{GLift}(k,2)$ is trivial. Then, by \Cref{klein-cohomology}, there exist $U$ and $V$ in $M_2(k)$ such that 
    \[N_\rho(U) = x^2E_{12},\qquad N_\mu(V) = y^2E_{12},\qquad N_\mu(U) = N_\rho(V)\] in $M_2(k)$, that is, letting $U=(u_{ij})$ and $V=(v_{ij})$, 
	\[
	\begin{bmatrix}
		x^2u_{21} & x^2u_{11} + x^4u_{21} + x^2u_{22} \\
		0 & x^2u_{21}
	\end{bmatrix}
	=
	\begin{bmatrix}
		0 & x^2 \\
		0 & 0
	\end{bmatrix},
	\]
	\[
	\begin{bmatrix}
		y^2v_{21} & y^2v_{11} + y^4v_{21} + y^2v_{22} \\
		0 & y^2v_{21}
	\end{bmatrix}=
	\begin{bmatrix}
		0 & y^2 \\
		0 & 0
	\end{bmatrix},
	\]
	\[
	\begin{bmatrix}
		y^2u_{21} & y^2u_{11} + y^4u_{21} + y^2u_{22} \\
		0 & y^2u_{21}
	\end{bmatrix}
	=
	\begin{bmatrix}
		x^2v_{21} & x^2v_{11} + x^4v_{21} + x^2v_{22} \\
		0 & x^2v_{21}
	\end{bmatrix}.
	\]
	It remains to show that no such $U$ and $V$ exist. Indeed, if they existed, then
    \[
	u_{21} = 0 = v_{21},
	\]
	\[
	u_{11} + u_{22} = 1 = v_{11} + v_{22},\]
	\[
	y^2u_{11} + y^2u_{22} = x^2v_{11} + x^2v_{22},
	\]
    which would imply $x^2=y^2$ and hence $x=y$, a contradiction.
\end{proof}

\begin{lemma}\label{restriction-kills-negligible}
	Let $F$ be a field of characteristic different from $2$, let $k$ be a field of characteristic $2$ such that $|k|>2$, let $W\subset k$ be a finite subgroup such that $|W|>2$, let $H\subset U_3(k)$ be the finite subgroup 
    \[H\coloneqq\begin{pmatrix}
        1 & x & z \\
        0 & 1 & y \\
        0 & 0 & 1
    \end{pmatrix},\]
    where $x,y\in W$ and $z\in \ang{W\cdot W}$, and let $Z\subset H$ be the center of $H$, that is, the subgroup of $H$ defined by $x=y=0$. The restriction map
	\[
	H^2(H, M_3(k)) \to H^2(Z, M_3(k))
	\]
	sends $H^2_{\on{neg},F}(H, M_3(k))$ to zero.
\end{lemma}

\begin{proof}
    Let $M\coloneqq M_3(k)$. By \Cref{basic-lemma}(3), we may assume that $F$ contains all roots of unity of $2$-power order. Since $H$ is finite, the conclusion will follow from \Cref{thm-negligible} once we show that $\on{res}^H_Z\on{cor}^S_H(A \cup \chi) = 0$ in $H^2(Z,M)$ for all subgroups $S \subset H$, for all $A \in M^S$ and all $\chi \in H^2(S, \Z)$. 

    Choose a subgroup $S \subset H$, an element $A \in M^S$, and an element $\chi \in H^2(S, \Z)$. Letting $\partial\colon H^1(S,\Q/\Z)\to H^2(S,\Z)$ be the connecting homomorphism associated to the short exact sequence of $S$-modules
    \[\begin{tikzcd}
        0 \arrow[r] & \Z \arrow[r] & \Q \arrow[r] & \arrow[r] \Q/\Z & 0,  
    \end{tikzcd}\]
    we have $\chi=\partial(u)$ for a unique character $u\colon S\to \Q/\Z$. Since $Z$ is an elementary abelian $2$-group, there exists a homomorphism $v\colon Z\to \Q/\Z$ which extends the restriction of $u$ to $S\cap Z$.
Then the map $\tilde{u}\colon SZ \to \Q/\Z$ defined by $sz \mapsto u(s) + v(z)$ for all $s\in S$ and $z\in Z$ is a well-defined character which extends $u$. Letting $\tilde{\chi}$ be the image of $\tilde{u}$ in $H^2(SZ,\Z)$, we deduce that $\on{res}^{SZ}_S(\tilde{\chi})=\chi$. By the projection formula, we have
\[\on{cor}^S_H(A\cup\chi)=\on{cor}^{SZ}_H(\on{cor}^S_{SZ}(A\cup\chi))=\on{cor}^{SZ}_H(N_{SZ/Z}(A)\cup\tilde{\chi}).\]
Therefore, replacing $S$ by $SZ$, we may assume that $Z \subset S$.

	Note that $Z \cap \sigma S \sigma^{-1} = Z$ for every $\sigma \in H$. Hence, by the double coset formula
	\begin{align*}
		\mathrm{res}^H_Z \, \mathrm{cor}^S_H(A \cup \chi) =&
		\sum_\sigma \sigma_* (A \cup \mathrm{res}_Z^S (\chi)) \\
		=&
		\sum_\sigma \sigma_*(A) \cup \sigma_*(\mathrm{res}_Z^S (\chi)) \\
		=&
		\sum_\sigma \sigma_*(A) \cup \mathrm{res}_Z^S (\chi) \\
		=&
		N_{H/S}(A) \cup \mathrm{res}_Z^S (\chi)
	\end{align*}
	in $H^2(Z,M)$, where $\sigma$ runs over a set of representatives of $H / S$. In order to conclude, it remains to show that $N_{H/S}(A)\cup\on{res}^S_Z(\chi)=0$.
	
	Let $W^\times\coloneqq W\setminus\{0\}$. For every $x\in W^\times$, we define $\sigma_{12}(x)\coloneqq I+xE_{12}\in H$ and $\sigma_{23}(x)\coloneqq I+xE_{23}\in H$. For all $x\in W^\times$, we have 
	\[M^{\sigma_{12}(x)} = 
	\begin{bmatrix}
		a & b & c \\
		0 & a & 0 \\
		0 & d & e
	\end{bmatrix},
	\]
	which is independent of the choice of $x\in W^\times$. Thus, for all $x, y \in W^\times$, we have
	\begin{equation}\label{n-sigma-12}
	N_{\sigma_{12}(x)}(M^{\sigma_{12}(y)}) = N_{\sigma_{12}(x)}(M^{\sigma_{12}(x)}) = 0.
	\end{equation}
	Similarly,
	\begin{equation}\label{n-sigma-23}
	N_{\sigma_{23}(x)}(M^{\sigma_{23}(y)}) = N_{\sigma_{23}(x)}(M^{\sigma_{23}(x)}) = 0.
	\end{equation}
	We split the proof that $N_{H/S}(A)\cup\on{res}^Z_S(\chi)=0$ in five cases.
	
	(i) Suppose first that $\sigma_{12}(x) \notin S$ for all $x\in W^\times$. Choose two distinct $x, y \in W^\times$, and let $K$ be the subgroup of $H$ generated by $S$, $\sigma_{12}(x)$, and $\sigma_{12}(y)$. Since $Z \subset S$ and the group $H / Z$ is abelian, we see that $S$ is normal in $K$ and $K / S$ is a Klein group generated by the cosets of $\sigma_{12}(x)$ and $\sigma_{12}(y)$. It follows from (\ref{n-sigma-12}) that
	\[
	N_{K/S}(A) = N_{\sigma_{12}(x)} (N_{\sigma_{12}(y)}(A)) = 0,
	\]
	and hence
	\[
	N_{H/S}(A) = N_{H/K}(N_{K/S}(A)) = 0.
	\]
    
    (ii) Suppose that $\sigma_{23}(x) \notin S$ for all $x\in W^\times$. The conclusion follows as in case (i), replacing $\sigma_{12}$ by $\sigma_{23}$ and (\ref{n-sigma-12}) by (\ref{n-sigma-23}).

	(iii) Suppose now that there are $x, y \in W^\times$ such that $\sigma_{12}(x) \notin S$ and $\sigma_{12}(y) \in S$.	Let $K$ be the subgroup of $H$ generated by $S$ and $\sigma_{12}(x)$. Then $K / S$ is a cyclic group generated by the coset of $\sigma_{12}(x)$. Since $A \in M^S \subset M^{\sigma_{12}(y)}$, it follows from (\ref{n-sigma-12}) that
	\[
	N_{K/S}(A) = N_{\sigma_{12}(x)}(A) = 0
	\]
	and hence
	\[
	N_{H/S}(A) = N_{H/K}(N_{K/S}(A)) = 0.
	\]
	
	(iv) Suppose now that there are $x, y \in W^\times$ such that $\sigma_{23}(x) \notin S$ and $\sigma_{23}(y) \in S$.	We conclude as in case (iii), replacing $\sigma_{12}$ by $\sigma_{23}$ and (\ref{n-sigma-12}) by (\ref{n-sigma-23}).
	
	(v) Finally, suppose that $\sigma_{12}(x)$ and $\sigma_{23}(x)$ belong to $S$ for all $x \in W^\times$.	In this case, $S = H$. Then $\mathrm{res}^S_Z(\chi) = 0$ since $Z \subset [S, S]$.
\end{proof}

\begin{proof}[Proof of \Cref{mainthm} for $p=2$, $|k|>2$ and $n\geq 3$]
	By \Cref{blift-implies-glift}, it suffices to show that $\on{GLift}(k,n)$ is not negligible over $F$ for all $n\geq 3$. By \Cref{reduce-dimension}, we may assume that $n=3$ and, by \Cref{basic-lemma}(3), we may suppose that $F$ contains all roots of unity of $2$-power order.
    
    Let $W\subset k$ be a finite subgroup such that $|W|>2$, for example a Klein subgroup. 
    Let $H\subset \on{GL}_3(k)$ and $Z\subset H$ be the corresponding finite subgroups in the statement of \Cref{restriction-kills-negligible}, and let $\alpha\in H^2(H,M_3(k))$ be the restriction of the class of $\on{GLift}(k,3)$ to $H$. By \Cref{basic-lemma}(2), it suffices to show that $\alpha$ is not negligible over $F$. By \Cref{restrict-k-bigger-f2}, the restriction of $\alpha$ in $H^2(Z,M_3(k))$ is not zero. By \Cref{restriction-kills-negligible}, the subgroup $H^2_{\on{neg},F}(H,M_3(k))$ restricts to zero in $H^2(Z,M_3(k))$. Thus $\on{res}^H_Z(\alpha)$ is not negligible over $F$. By \Cref{basic-lemma}(2), we conclude that $\alpha$ is not negligible over $F$, as desired. 
\end{proof}

\section{Proof of Theorem \ref{mainthm} for \texorpdfstring{$|k|=2$}{|k|=2} and \texorpdfstring{$n\geq 5$}{n≥5}}\label{section-6}

\subsection{Notation} Throughout this section, we let $G\coloneqq \on{GL}_5(\F_2)$, $U\coloneqq U_5(\F_2)$, and $M\coloneqq M_5(\F_2)$. The group $G$ acts on $M$ by matrix conjugation. For every $A\in M$, we let $G_A$ be the stabilizer of $A$ in $G$. For every subgroup $H\subset G$, we define
\[\varphi_H\colon M^H\otimes H^2(H,\Z)\xrightarrow{\cup} H^2(H,M)\xrightarrow{\on{cor}}H^2(G,M).\]

For every subgroup $H\subset U_5$ and all $1\leq i\leq j\leq 5$, we let $u_{ij}\colon H\to \Q/\Z$ be the composition of the $(i,j)$-th coordinate function $H\to \Z/2\Z$ and the inclusion $\Z/2\Z\hookrightarrow\Q/\Z$. The function $u_{ij}$ is not necessarily a group homomorphism. If it is a homomorphism, then it defines an element $u_{ij}\in H^1(H,\Q/\Z)$, and we let $\chi_{ij}\coloneqq \partial(u_{ij})\in H^2(H,\Z)$, where $\partial\colon H^1(H,\Q/\Z)\to H^2(H,\Z)$ is the connecting map associated to the sequence 
\[\begin{tikzcd}
    0 \arrow[r] & \Z \arrow[r] & \Q \arrow[r] & \Q/\Z \arrow[r] & 0.
\end{tikzcd}
\]

For every $A\in M$, we let $p_A(x),q_A(x)\in \F_2[x]$ be the characteristic polynomial and the minimal polynomial of $A$, respectively. Observe that $\on{deg}(p_A(x))=5$, that $q_A(x)$ divides $p_A(x)$, and that $p_A(x)$ and $q_A(x)$ have the same irreducible factors.

We let $\Pi\colon S_5\to G$ be the homomorphism which sends a permutation $\sigma\in S_5$ to the corresponding permutation matrix $\Pi(\sigma)$.

\subsection{Projection formula arguments}

We collect lemmas that will be invoked repeatedly in what follows. Their proofs use the projection formula \cite[Proposition 1.5.3(iv)]{neukirch2008cohomology}.

\begin{lemma}\label{sylow}
	Let $H\subset G$ be a subgroup, let $P\subset H$ be a $2$-Sylow subgroup, let $A\in M^H$, let $\chi\in H^2(H,\Z)$, and let $\chi'\coloneqq \on{res}^H_P(\chi)\in H^2(P,\Z)$. If $\varphi_P(A\otimes\chi')=0$, then $\varphi_H(A\otimes\chi)=0$. In particular, if $\varphi_P=0$, then $\varphi_H=0$.
\end{lemma}

\begin{proof}
	Since $2M=0$ and $[H:P]$ is odd, we have $N_{H/P}(A)=[H:P]A=A$. By the projection formula \[\on{cor}^H_G(A\cup\chi)=\on{cor}^H_G(N_{H/P}(A)\cup\chi)=\on{cor}^H_G(\on{cor}^P_H(A\cup\chi'))=\on{cor}^P_G(A\cup\chi')=0.\qedhere\]
\end{proof}

\begin{lemma}\label{u5}
	We have $\varphi_U=0$.
\end{lemma}

\begin{proof}
	We have $M^U=\ang{E_{15}}$ and
	\[G_{E_{15}}=\begin{pmatrix}
		1 & * & * & * & * \\
		0 & * & * & * & * \\
		0 & * & * & * & * \\
		0 & * & * & * & * \\
		0 & 0 & 0 & 0 & 1
	\end{pmatrix}.\]
    We have \[H^1(U,\Q/\Z)=\F_2\cdot u_{12}\oplus\F_2\cdot u_{23}\oplus\F_2\cdot u_{34}\oplus\F_2\cdot u_{45},\]
    and hence \[H^2(U,\Z)=\F_2\cdot \chi_{12}\oplus\F_2\cdot \chi_{23}\oplus\F_2\cdot \chi_{34}\oplus\F_2\cdot \chi_{45}.\] 
    If $\chi\in\{\chi_{12},\chi_{23}\}$, define a subgroup $U\subset K\subset G$ as
	\[K\coloneqq \begin{pmatrix}
		1 & * & * & * & * \\
		0 & 1 & * & * & * \\
		0 & 0 & 1 & * & * \\
		0 & 0 & 0 & * & * \\
		0 & 0 & 0 & * & *
	\end{pmatrix}.\]
    Then $u_{12}$ and $u_{23}$ extend to $K$, and hence $\chi$ is the restriction of some $\chi'\in H^2(K,\Z)$. Observe that $M^K\subset M^U$ and that $E_{15}$ is not $K$-invariant because $K$ is not contained in $G_{E_{15}}$. We deduce that $M^K=0$, so that in particular $\on{cor}^U_K(E_{15})=0$ and therefore
	\[\on{cor}_G^{U}(E_{15}\cup \chi)=\on{cor}_G^K(\on{cor}^U_K(E_{15}\cup\chi))=\on{cor}_G^K(N_{U/K}(E_{15})\cup\chi')=0.\]
	If $\chi\in\{\chi_{34},\chi_{45}\}$, a similar argument, replacing $K$ by the subgroup
	\[K'\coloneqq \begin{pmatrix}
		* & * & * & * & * \\
		* & * & * & * & * \\
		0 & 0 & 1 & * & * \\
		0 & 0 & 0 & 1 & * \\
		0 & 0 & 0 & 0 & 1
	\end{pmatrix},\]
    again shows that $\on{cor}_G^U(E_{15}\cup \chi)=0$. Thus $\varphi_U=0$, as desired.
\end{proof}

\begin{lemma}\label{cor-comes-from-u5}
	Let $H\subset U$ be a subgroup. For all $1\leq i\leq 4$ and all $A\in M^H$, we have $\varphi_H(A\otimes\chi_{i,i+1})=0$. 
\end{lemma}

\begin{proof}
	The $\chi_{i,i+1}\in H^2(H,\Z)$ extend to elements of $H^2(U,\Z)$. By \Cref{u5} and the projection formula, we have \[\on{cor}^H_G(A\cup\chi_{i,i+1})=\on{cor}^U_G(\on{cor}^H_U(A\cup\chi_{i,i+1}))=\on{cor}^U_G(N_{U/H}(A)\cup\chi_{i,i+1})=0.\qedhere\]
\end{proof}	

\begin{lemma}\label{cor-other-subgroups}
	For $i=2,3$, let $V_i\coloneqq \on{Ker}[u_{i,i+1}\colon U\to \F_2]$. 
	
	(1) Let $H\subset V_2$ be a subgroup, and let $\chi \in H^2(H,\Z)$ be either $\partial(u_{13})$ or $\partial(u_{24})$. For all $A\in M^H$, we have $\varphi_H(A\otimes\chi)=0$.
	
	(2) Let $H\subset V_3$ be a subgroup, and let $\chi \in H^2(H,\Z)$ be either $\partial(u_{24})$ or $\partial(u_{35})$. For all $A\in M^H$, we have $\varphi_H(A\otimes\chi)=0$.
\end{lemma}

\begin{proof}
	We first show that $\varphi_{V_i}=0$ for $i=2,3$. We have $M^{V_i}=\ang{E_{15}}=M^U$.
	By the projection formula, for every $\chi\in H^2(V_i,\Z)$ we have
	\[\on{cor}^{V_i}_G(E_{15}\cup\chi)=\on{cor}^U_G(\on{cor}^{V_i}_U(E_{15}\cup\chi))=\on{cor}^U_G(E_{15}\cup\on{cor}^{V_i}_U(\chi)).\]
	Now \Cref{u5} implies that $\varphi_{V_i}=0$, as claimed.
	
	The coordinate maps $u_{13},u_{24}\colon V_1\to \F_2$ and $u_{24},u_{35}\colon V_2\to \F_2$ are group homomorphisms. Thus $\chi$ as in (1) and (2) is well defined and is the restriction of some $\chi'\in H^2(V_i,\Z)$. For every $A\in M^H$, the vanishing of $\varphi_{V_i}$ implies 
    \[\on{cor}^{V_i}_G(N_{H/V_i}(A)\cup\chi')=\varphi_{V_i}(N_{H/V_i}(A)\otimes\chi')=0,\]
    and hence by the projection formula
	\[\on{cor}^H_G(A\cup\chi)=\on{cor}^{V_i}_G(\on{cor}^H_{V_i}(A\cup\chi))=\on{cor}^{V_i}_G(N_{H/V_i}(A)\cup\chi')=0.\qedhere\]
\end{proof}

\begin{lemma}\label{identity-5x5}
    Let $I\in M$ be the identity matrix. For every subgroup $H\subset G$ and every $\chi\in H^2(H,\Z)$, we have $\varphi_H(I\otimes\chi)=0$.
\end{lemma}

\begin{proof}
    Recall that $\on{GL}_n(\F_2)=\on{SL}_n(\F_2)$ is equal to its derived subgroup for all $n\geq 3$; see for example \cite[Theorem 24.17]{malle2011linear}. Thus $G=[G,G]$, and hence $H^2(G,\Z)=0$. By the projection formula, we conclude that
    \[\on{cor}^H_G(I\cup\chi)=I\cup N_{G/H}(\chi)=0.\qedhere\]
\end{proof}

\subsection{The case when \texorpdfstring{$A$}{A} is not conjugate to a Jordan block}

\begin{prop}\label{all-phi-but-one-vanish}
    Suppose that $A\in M$ is not conjugate to a $5\times 5$ Jordan block. Then $\varphi_{G_A}(A\otimes \chi)=0$ for all $\chi \in H^2(G_A,\Z)$.
\end{prop}

We will prove the conclusion of \Cref{all-phi-but-one-vanish} by a case-by-case analysis:
\begin{itemize}
    \item[(1)] $A$ is diagonalizable over $\F_2$ (\Cref{lemma-diagonalizable}),
    \item[(2)] $p_A(x)$ is square-free (\Cref{lemma-cyclic}),
    \item[(3)] $A$ is not diagonalizable, but it admits a Jordan form over $\F_2$ (\Cref{lemma-jordan}),
    \item[(4)] $p_A(x)$ is not square-free and does not split over $\F_2$ (\Cref{lemma-other}).
\end{itemize}

By \Cref{modulation}, for the proof of \Cref{all-phi-but-one-vanish} it suffices to consider one $A\in M$ for each $G$-orbit. Moreover, letting $M^0\subset M$ be the $G$-submodule of trace-zero matrices, for every $A\in M$ one of $A$ and $I+A$ belongs to $M^0$, and hence by \Cref{identity-5x5} we may assume that $A\in M^0$.

\begin{lemma}\label{lemma-diagonalizable}
	If $A\in M$ is diagonalizable over $\F_2$, then $\phi_{G_A}=0$.
\end{lemma}

\begin{proof}
	By \Cref{identity-5x5}, we may assume that $A\in M^0$. We have \[G_A\cong \on{GL}_d(\F_2)\times \on{GL}_{5-d}(\F_2)\] for some $0\leq d\leq 2$. For all $r\geq 3$ the group $\on{GL}_r(\F_2)$ is equal to its derived subgroup, and hence $H^2(\on{GL}_r(\F_2),\Z)=H^1(\on{GL}_r(\F_2),\Q/\Z)=0$. Thus, if $d=0,1$, we have $H^2(G_A,\Z)=0$ and hence $\varphi_{G_A}=0$ in this case. When $d=2$, up to conjugation 
	\[A=\begin{bmatrix}
		0 & 0 & 0 & 0 & 0 \\
		0 & 0 & 0 & 0 & 0 \\
		0 & 0 & 0 & 0 & 0 \\
		0 & 0 & 0 & 1 & 0 \\
		0 & 0 & 0 & 0 & 1
	\end{bmatrix},\qquad 
	G_A=\begin{pmatrix}
		* & * & * & 0 & 0 \\
		* & * & * & 0 & 0 \\
		* & * & * & 0 & 0 \\
		0 & 0 & 0 & * & * \\
		0 & 0 & 0 & * & *
	\end{pmatrix}.\]
	Consider the following $2$-Sylow subgroup of $G_A$:
	\[P\coloneqq \begin{pmatrix}
		1 & * & * & 0 & 0 \\
		0 & 1 & * & 0 & 0 \\
		0 & 0 & 1 & 0 & 0 \\
		0 & 0 & 0 & 1 & * \\
		0 & 0 & 0 & 0 & 1
	\end{pmatrix}\cong U_3(\F_2)\times U_2(\F_2).\]
	By \Cref{sylow}, it suffices to show that $\varphi_P=0$. As \[H^1(P,\Q/\Z)=\F_2\cdot u_{12}\oplus \F_2\cdot u_{23}\oplus \F_2\cdot u_{45},\] every character of $P$ extends to $U_5$, and the conclusion follows from \Cref{cor-comes-from-u5}.
\end{proof}

\begin{lemma}\label{lemma-cyclic}
	For every $A\in M$ such that $p_A(x)$ is square-free, we have $\phi_{G_A}=0$. 
\end{lemma}

\begin{proof}
We view $M$ as a non-commutative $\F_2$-algebra, and for every $A\in M$ we let $Z(A)\subset M$ be the centralizer $\F_2$-subalgebra of $A$. Then $G_A=Z(A)^\times$. 

Suppose that $p_A(x)$ is square-free. Then $p_A(x)=q_A(x)$, and hence $A$ admits a cyclic basis. It follows that $Z(A)$ is equal to the $\F_2$-subalgebra generated by $A$, and hence $Z(A)\cong \F_2[x]/(p_A(x))$ as $\F_2$-algebras. In particular, $G_A\cong \F_2[x]/(p_A(x))^\times$. Because $p_A(x)$ is square-free, $\F_2[x]/(p_A(x))\cong F_1\times \dots \times F_d$, where $F_i/\F_2$ is a finite field extension for all $1\leq i\leq d$. Thus $G_A\cong F_1^\times\times\dots\times F_d^\times$ has odd order. We conclude that $H^2(G_A,\Z)[2]=0$, and hence in particular $\varphi_{G_A}=0$.
\end{proof}

\begin{lemma}\label{lemma-jordan}
	Suppose that $A\in M$ is not diagonalizable over $\F_2$, that $p_A(x)$ splits as a product of linear factors in $\F_2[x]$, and that $A$ is not conjugate to a $5\times 5$ Jordan block. Then $\varphi_{G_A}(A\otimes \chi)=0$ for all $\chi \in H^2(G_A,\Z)$.
\end{lemma}

\begin{proof}
	By \Cref{identity-5x5}, we may assume that the trace of $A$ is zero. Up to conjugation, we may assume that $A$ is in normal Jordan form. We let $J_r(\lambda)$ be the $r\times r$ Jordan block with eigenvalue $\lambda_i\in \F_2$. More generally, we let $J_{r_1}(\lambda_1)\oplus \dots\oplus J_{r_d}(\lambda_d)$ be the matrix in Jordan form with $i$-th Jordan block of size $r_i\geq 1$ and eigenvalue $\lambda_i\in \F_2$.
	
	(i)    If $A=J_5(0)$, there is nothing to prove.

	(ii) If $A=J_4(0)\oplus J_1(0)$, then
	\[G_A=
	\begin{pmatrix}
		1 & a & b & c & d \\
		0 & 1 & a & b & 0 \\
		0 & 0 & 1 & a & 0 \\
		0 & 0 & 0 & 1 & 0 \\
		0 & 0 & 0 & e & 1
	\end{pmatrix}.
	\]
	We replace $A$ by its conjugate by the permutation matrix $\Pi(45)$. Then
	\[
	A=\begin{bmatrix}
		0 & 1 & 0 & 0 & 0 \\
		0 & 0 & 1 & 0 & 0 \\
		0 & 0 & 0 & 0 & 1 \\
		0 & 0 & 0 & 0 & 0 \\
		0 & 0 & 0 & 0 & 0
	\end{bmatrix},\qquad
	G_A=\begin{pmatrix}
		1 & a & b & d & c \\
		0 & 1 & a & 0 & b \\
		0 & 0 & 1 & 0 & a \\
		0 & 0 & 0 & 1 & e \\
		0 & 0 & 0 & 0 & 1
	\end{pmatrix}.
	\]
	We have $I+E_{15}=[I+E_{14},I+E_{45}]$. We deduce that $G_A^{\on{ab}}=\Z/4\Z\times \Z/2\Z\times \Z/2\Z$ and $H^1(G_A,\Q/\Z)=(\Z/4\Z)\cdot u\oplus \F_2\cdot u_{14}\oplus \F_2\cdot u_{45}$, where $u\colon G_A\to \Z/4\Z\subset \Q/\Z$ is defined as follows. Let $C\subset U_3$ be the subgroup generated by the order $4$ element $I+E_{12}+E_{23}$. We have an isomorphism $\rho\colon C\xrightarrow{\sim}\Z/4\Z$ which sends $I+E_{12}+E_{23}$ to $1+4\Z$. Then $u$ is the composite of the projection onto the top-left $3\times 3$ square (whose image is equal to $C$) and $\rho$. Therefore \[H^2(G_A,\Z)=(\Z/4\Z)\chi\oplus \F_2\chi_{14}\oplus \F_2\chi_{45},\] where $\chi\coloneqq \partial(u)$. 
    
    We first show that $\on{cor}^{G_A}_G(A\cup\chi)=0$. Let
	\[K\coloneqq\begin{pmatrix}
		1 & a & b & d & c \\
		0 & 1 & a & 0 & f \\
		0 & 0 & 1 & 0 & a \\
		0 & 0 & 0 & 1 & e \\
		0 & 0 & 0 & 0 & 1
	\end{pmatrix}.
	\]
	Then $u$ extends to $u'\colon K\to \Z/4\Z$, with the same definition. We let $\chi'\coloneqq \partial(u')$. Let $\sigma\coloneqq I+E_{25}\in K$. Then $K$ is the internal semidirect product $G_A\rtimes\ang{\sigma}$. It follows that $N_{K/G_A}(A)=N_\sigma(A)=E_{15}$. Now the projection formula implies that $\on{cor}^{G_A}_G(A\cup\chi)=\on{cor}^K_G(E_{15}\cup\chi')$. Let
	\[L\coloneqq G_{E_{15}}=\begin{pmatrix}
		1 & * & * & * & * \\
		0 & * & * & * & * \\
		0 & * & * & * & * \\
		0 & * & * & * & * \\
		0 & 0 & 0 & 0 & 1
	\end{pmatrix}.
	\]
	By the projection formula $\on{cor}^K_L(E_{15}\cup\chi')=E_{15}\cup\on{cor}^K_L(\chi')$. We have \[H^1(L,\Q/\Z)=\F_2\cdot u_{12}\oplus \F_2\cdot u_{45},\] and so $\varphi_L=0$ by \Cref{cor-comes-from-u5}. We conclude that $\on{cor}^K_G(E_{15}\cup\chi')=0$, and hence in particular $\on{cor}^{G_A}_G(A\cup\chi)=0$.
    
	The fact that $\on{cor}^{G_A}_G(A\cup\chi_{45})=0$ follows from \Cref{cor-comes-from-u5}. Finally, in order to deal with $\chi_{14}$, we further conjugate $A$ by $\Pi(243)$. Then $G_A$ is sent to
	\[
	\begin{pmatrix}
		1 & d & a & b & c \\
		0 & 1 & 0 & 0 & e \\
		0 & 0 & 1 & a & b \\
		0 & 0 & 0 & 1 & a \\
		0 & 0 & 0 & 0 & 1
	\end{pmatrix}
	\]
	and $u_{14}$ is sent to $u_{12}$. Now \Cref{cor-comes-from-u5} implies that $\on{cor}^{G_A}_G(A\cup\chi_{14})=0$.
	
	(iii) If $A=J_3(0)\oplus J_2(0)$, then
	\[G_A=
	\begin{pmatrix}
		1 & a & b & c & d \\
		0 & 1 & a & 0 & c \\
		0 & 0 & 1 & 0 & 0 \\
		0 & e & f & 1 & g \\
		0 & 0 & e & 0 & 1
	\end{pmatrix}.
	\]
	Conjugate $A$ by $\Pi(2354)$ to get
	\[ 
    A=\begin{bmatrix}
		0 & 0 & 1 & 0 & 0 \\
		0 & 0 & 0 & 1 & 0 \\
		0 & 0 & 0 & 0 & 1 \\
		0 & 0 & 0 & 0 & 0 \\
		0 & 0 & 0 & 0 & 0
	\end{bmatrix},\qquad
	G_A=
	\begin{pmatrix}
		1 & c & a & d & b \\
		0 & 1 & e & h & f \\
		0 & 0 & 1 & c & a \\
		0 & 0 & 0 & 1 & e \\
		0 & 0 & 0 & 0 & 1
	\end{pmatrix}.
	\]
	Then $[G_A,G_A]$ contains 
    \begin{align*}
        I+E_{14} &=[I+E_{12}+E_{34},I+E_{13}+E_{35}], \\
        I+E_{25} &=[I+E_{23}+E_{45},I+E_{13}+E_{35}], \\
        I+E_{15} &=[I+E_{14},I+E_{23}+E_{45}], \\
        I+E_{13}+E_{24}+E_{35} &=[I+E_{12}+E_{34},I+E_{23}+E_{45}].
    \end{align*}
	Thus the abelianization $G_A^{\on{ab}}$ may be described as
	\[
	\begin{pmatrix}
		1 & c & a & \square & \square \\
		0 & 1 & e & h & \square \\
		0 & 0 & 1 & c & a \\
		0 & 0 & 0 & 1 & e \\
		0 & 0 & 0 & 0 & 1
	\end{pmatrix}\quad\text{modulo}\quad \begin{pmatrix}
		1 & 0 & 1 & \square & \square \\
		0 & 1 & 0 & 1 & \square \\
		0 & 0 & 1 & 0 & 1 \\
		0 & 0 & 0 & 1 & 0 \\
		0 & 0 & 0 & 0 & 1
	\end{pmatrix},
	\]
	where the boxes indicate that the corresponding entries are missing. Indeed, $G_A^{\on{ab}}$ is a quotient of this group, and on the other hand a simple computation shows that this group is abelian and in fact every element has order $2$. Thus $G_A^{\on{ab}}\cong (\Z/2\Z)^3$. In fact, we have an isomorphism
	\[G_A^{\on{ab}}\xrightarrow{\sim}(\Z/2\Z)^3,\quad \begin{pmatrix}
		1 & c & a & \square & \square \\
		0 & 1 & e & h & \square \\
		0 & 0 & 1 & c & a \\
		0 & 0 & 0 & 1 & e \\
		0 & 0 & 0 & 0 & 1
	\end{pmatrix}\mapsto (c,e,a+h+ce).\]
	In particular, $H^1(G_A,\Q/\Z)=\F_2\cdot u_{12}\oplus \F_2\cdot u_{23}\oplus \F_2\cdot u$, where 
	\[u\colon G_A\to \Z/2\Z,\qquad \begin{pmatrix}
		1 & c & a & d & b \\
		0 & 1 & e & h & f \\
		0 & 0 & 1 & c & a \\
		0 & 0 & 0 & 1 & e \\
		0 & 0 & 0 & 0 & 1
	\end{pmatrix}\mapsto a+h+ce.\]    
	We define $\chi\coloneqq \partial(u)\in H^2(G_A,\Z)$. In view of \Cref{cor-comes-from-u5}, it suffices to show that $\on{cor}^{G_A}_G(A\cup\chi)=0$. For this, define the subgroup
	\[K\coloneqq
	\begin{pmatrix}
		1 & c & a & d & b \\
		0 & 1 & e & h & f \\
		0 & 0 & 1 & c & g \\
		0 & 0 & 0 & 1 & e \\
		0 & 0 & 0 & 0 & 1
	\end{pmatrix}.    
	\]
	Then $K$ is the internal semidirect product $G_A\rtimes\ang{\sigma}$, where $\sigma\coloneqq I+E_{35}$. Observe that $u$ extends to a homomorphism $u'\colon K\to \Z/2\Z$, given by the same formula. Let $\chi'\coloneqq \partial(u')$. We have $N_{K/G_A}(A)=N_\sigma(A)=E_{15}$, and hence by the projection formula $\on{cor}^{G_A}_K(A\cup\chi)=E_{15}\cup \chi'$. This reduces us to showing that $\on{cor}^K_G(E_{15}\cup\chi')=0$. As in (ii), let 
	\[L\coloneqq G_{E_{15}}=\begin{pmatrix}
		1 & * & * & * & * \\
		0 & * & * & * & * \\
		0 & * & * & * & * \\
		0 & * & * & * & * \\
		0 & 0 & 0 & 0 & 1
	\end{pmatrix}.
	\]
	By the projection formula $\on{cor}^K_L(E_{15}\cup\chi')=E_{15}\cup\on{cor}^K_L(\chi')$. We have \[H^1(L,\Q/\Z)=\F_2\cdot u_{12}\oplus \F_2\cdot u_{45},\] and hence $\varphi_L=0$ by \Cref{cor-comes-from-u5}. In particular, \[\on{cor}^K_G(E_{15}\cup\chi')=\on{cor}^L_G(\on{cor}^K_L(E_{15}\cup\chi'))=0,\]
    as desired.

	(iv) If $A=J_3(0)\oplus J_1(0)\oplus J_1(0)$, then
	\[G_A=
	\begin{pmatrix}
		1 & a & b & c & d \\
		0 & 1 & a & 0 & 0 \\
		0 & 0 & 1 & 0 & 0 \\
		0 & 0 & e & f & g \\
		0 & 0 & h & i & j
	\end{pmatrix}.
	\]
	We conjugate $A$ by $\Pi(35)$. Then $G_A$ is sent to
	\[
	\begin{pmatrix}
		1 & a & d & c & b \\
		0 & 1 & 0 & 0 & a \\
		0 & 0 & j & i & h \\
		0 & 0 & g & f & e \\
		0 & 0 & 0 & 0 & 1
	\end{pmatrix}.
	\]
	A $2$-Sylow subgroup of $G_A$ is
	\[
	P=\begin{pmatrix}
		1 & a & d & c & b \\
		0 & 1 & 0 & 0 & a \\
		0 & 0 & 1 & i & h \\
		0 & 0 & 0 & 1 & e \\
		0 & 0 & 0 & 0 & 1
	\end{pmatrix}.
	\]
	We have
    \begin{align*}
        I+E_{15}=[I+E_{14},I+E_{45}], \\
        I+E_{35}=[I+E_{34},I+E_{45}], \\
        I+E_{14}=[I+E_{13},I+E_{34}]. 
    \end{align*}
	Then $G_A^{\on{ab}}\cong (\Z/2\Z)^4$, so that \[H^1(G_A,\Q/\Z)=\F_2\cdot u_{12}\oplus \F_2\cdot u_{13}\oplus\F_2\cdot u_{34}\oplus \F_2\cdot u_{45}.\] The conclusion follows from \Cref{cor-comes-from-u5} and \Cref{cor-other-subgroups}.
	
	(v) If $A=J_2(0)\oplus J_2(0)\oplus J_1(0)$, then
	\[G_A=
	\begin{pmatrix}
		a & b & c & d & e \\
		0 & a & 0 & c & 0 \\
		f & g & h & i & 0 \\
		0 & f & 0 & h & 0 \\
		0 & j & 0 & k & l
	\end{pmatrix}.
	\]
	Conjugate $A$ by $\Pi(2453)$ to get
	\[G_A=
	\begin{pmatrix}
		a & c & e & b & d \\
		f & h & 0 & g & i \\
		0 & 0 & l & j & k \\
		0 & 0 & 0 & a & c \\
		0 & 0 & 0 & f & h 
	\end{pmatrix}.
	\]
	A $2$-Sylow subgroup of $G_A$ is
	\[P\coloneqq
	\begin{pmatrix}
		1 & c & e & b & d \\
		0 & 1 & 0 & g & i \\
		0 & 0 & 1 & j & k \\
		0 & 0 & 0 & 1 & c \\
		0 & 0 & 0 & 0 & 1 
	\end{pmatrix}.
	\]
	The commutator subgroup $[G_A,G_A]$ contains $I+E_{15}$, $I+E_{14}$, $I+E_{25}$, $I+E_{35}$. It follows that $G_A^{\on{ab}}\cong (\Z/2\Z)^4$, so that \[H^1(G_A,\Q/\Z)=\F_2\cdot u_{12}\oplus\F_2\cdot u_{13}\oplus\F_2\cdot u_{34}\oplus\F_2\cdot u_{24}.\] The conclusion follows from \Cref{cor-comes-from-u5} and \Cref{cor-other-subgroups}.
	
	(vi) If $A=J_2(0)\oplus J_1(0)\oplus J_1(0)\oplus J_1(0)$, then up to conjugation $A=I+E_{15}$, in which case
	\[G_A=
	\begin{pmatrix}
		1 & * & * & * & * \\
		0 & * & * & * & * \\
		0 & * & * & * & * \\
		0 & * & * & * & * \\
		0 & 0 & 0 & 0 & 1
	\end{pmatrix}
	\]
	which contains $U$.  We have $H^1(G_A,\Q/\Z)=\F_2\cdot u_{12}\oplus \F_2\cdot u_{45}$. The conclusion follows from \Cref{cor-comes-from-u5}.
    
	(vii) If $A=J_4(1)\oplus J_1(0)$, then
	\[G_A=
	\begin{pmatrix}
		1 & a & b & c & 0 \\
		0 & 1 & a & b & 0 \\
		0 & 0 & 1 & a & 0 \\
		0 & 0 & 0 & 1 & 0 \\
		0 & 0 & 0 & 0 & 1
	\end{pmatrix}.
	\]  
	Let
	\[K\coloneqq 
	\begin{pmatrix}
		1 & a & b & c & d \\
		0 & 1 & a & b & e \\
		0 & 0 & 1 & a & 0 \\
		0 & 0 & 0 & 1 & 0 \\
		0 & 0 & 0 & 0 & 1
	\end{pmatrix}.
	\]
	Let $\sigma\coloneqq I+E_{15}$ and $\tau\coloneqq I+E_{25}$. Then $\ang{\sigma,\tau}\cong (\Z/2\Z)^2$ is a normal subgroup of $U_5$ which intersects $G_A$ trivially, and so $K= \ang{\sigma,\tau}\rtimes G_A$. In particular, every character of $G_A$ extends to $K$. By the projection formula, for every $\chi \in H^2(G_A,\Z)$ we have
	\[\on{cor}^{G_A}_G(A\cup\chi)=\on{cor}^K_G(\on{cor}^{G_A}_K(A\cup\chi))=\on{cor}^{G_A}_G(N_{K/H}(A)\cup\chi'),\]
	where $\chi'\in H^2(K,\Z)$ restricts to $\chi$ in $H^2(G_A,\Z)$. We have $N_\sigma(A)=E_{15}$, so that $N_{K/H}(A)=N_\tau(N_\sigma(A))=N_\tau(E_{15})=0$. Thus $\on{cor}^{G_A}_G(A\cup\chi)=0$ for all $\chi \in H^2(G_A,\Z)$.
	
	(viii) If $A=J_3(1)\oplus J_1(1)\oplus J_1(0)$, then
	\[G_A=
	\begin{pmatrix}
		1 & a & b & c & 0 \\
		0 & 1 & a & 0 & 0 \\
		0 & 0 & 1 & 0 & 0 \\
		0 & 0 & d & 1 & 0 \\
		0 & 0 & 0 & 0 & 1
	\end{pmatrix}.
	\]
	Conjugate $A$ by $\Pi(34)$ to get
	\[
    A=\begin{bmatrix}
		1 & 1 & 0 & 0 & 0 \\
		0 & 1 & 0 & 1 & 0 \\
		0 & 0 & 1 & 0 & 0 \\
		0 & 0 & 0 & 1 & 0 \\
		0 & 0 & 0 & 0 & 0
	\end{bmatrix},\qquad
	G_A=\begin{pmatrix}
		1 & a & c & b & 0 \\
		0 & 1 & 0 & a & 0 \\
		0 & 0 & 1 & d & 0 \\
		0 & 0 & 0 & 1 & 0 \\
		0 & 0 & 0 & 0 & 1
	\end{pmatrix}.
	\]
	Consider the subgroup 
	\[K\coloneqq\begin{pmatrix}
		1 & a & c & b & e \\
		0 & 1 & 0 & a & f \\
		0 & 0 & 1 & d & 0 \\
		0 & 0 & 0 & 1 & 0 \\
		0 & 0 & 0 & 0 & 1
	\end{pmatrix}.\]
	Let $\sigma\coloneqq I+E_{15}$ and $\tau\coloneqq I+E_{25}$. Then $K=\ang{\sigma,\tau}\rtimes G_A$, so that in particular every character of $G_A$ extends to $K$. We have $N_\sigma(A)=E_{15}$ and $N_\tau(E_{15})=0$, so that $N_{K/G_A}(A)=N_\tau(N_\sigma(A))=N_\tau(E_{15})=0$. 
	By the projection formula, for all $\chi\in H^2(G_A,\Z)$, letting $\chi'\in H^2(K,\Z)$ be a class restricting to $\chi$, we have
	\[\on{cor}^{G_A}_G(A\cup\chi)=\on{cor}^K_G(\on{cor}^{G_A}_K(A\cup\chi))=\on{cor}^K_G(N_{K/G_A}(A)\cup \chi')=0.\] 
	
	(ix) If $A=J_2(1)\oplus J_2(1)\oplus J_1(0)$, then
	\[G_A=
	\begin{pmatrix}
		a & b & c & d & 0 \\
		0 & a & 0 & c & 0 \\
		e & f & g & h & 0 \\
		0 & e & 0 & g & 0 \\
		0 & 0 & 0 & 0 & 1
	\end{pmatrix}.
	\]
	We conjugate $A$ by $\Pi(23)$. Then 
	\[
    A=\begin{bmatrix}
        
		1 & 0 & 1 & 0 & 0 \\
		0 & 1 & 0 & 1 & 0 \\
		0 & 0 & 1 & 0 & 0 \\
		0 & 0 & 0 & 1 & 0 \\
		0 & 0 & 0 & 0 & 0
    \end{bmatrix},
    \qquad
	G_A=\begin{pmatrix}
		a & c & b & d & 0 \\
		e & g & f & h & 0 \\
		0 & 0 & a & c & 0 \\
		0 & 0 & e & g & 0 \\
		0 & 0 & 0 & 0 & 1
	\end{pmatrix}.
	\]
	A $2$-Sylow subgroup of $G_A$ is given by
	\[P\coloneqq \begin{pmatrix}
		1 & c & b & d & 0 \\
		0 & 1 & f & h & 0 \\
		0 & 0 & 1 & c & 0 \\
		0 & 0 & 0 & 1 & 0 \\
		0 & 0 & 0 & 0 & 1
	\end{pmatrix}.\]
	By \Cref{sylow}, it suffices to show that $\on{cor}^P_G(A\cup\chi)=0$ for all $\chi\in H^2(P,\Z)$. Let 
	\[K\coloneqq \begin{pmatrix}
		1 & c & b & d & i \\
		0 & 1 & f & h & j \\
		0 & 0 & 1 & c & 0 \\
		0 & 0 & 0 & 1 & 0 \\
		0 & 0 & 0 & 0 & 1
	\end{pmatrix}.\]
	Then $K=\ang{\sigma,\tau}\rtimes P$, where $\sigma\coloneqq I+E_{15}$ and $\tau\coloneqq I+E_{25}$. Every character of $P$ extends to $K$, and hence by the projection formula it suffices to show that $N_{K/P}(A)=0$. We have $N_\sigma(A)=E_{15}$ and $N_\tau(E_{15})=0$, which together imply $N_{K/P}(A)=N_\tau(N_\sigma(A))=0$, as desired.
	
	(x) If $A=J_2(1)\oplus J_1(1)\oplus J_1(1)\oplus J_1(0)$, then
	\[G_A=
	\begin{pmatrix}
		1 & a & b & c & 0 \\
		0 & 1 & 0 & 0 & 0 \\
		0 & d & e & f & 0 \\
		0 & g & h & i & 0 \\
		0 & 0 & 0 & 0 & 1
	\end{pmatrix}.
	\]
	We conjugate $A$ by $\Pi(24)$, so that $G_A$ is replaced by
	\[
	G_A=\begin{pmatrix}
		1 & c & b & a & 0 \\
		0 & i & h & g & 0 \\
		0 & f & e & d & 0 \\
		0 & 0 & 0 & 1 & 0 \\
		0 & 0 & 0 & 0 & 1
	\end{pmatrix}.
	\]
	A $2$-Sylow subgroup of $G_A$ is
	\[
	P\coloneqq \begin{pmatrix}
		1 & c & b & a & 0 \\
		0 & 1 & h & g & 0 \\
		0 & 0 & 1 & d & 0 \\
		0 & 0 & 0 & 1 & 0 \\
		0 & 0 & 0 & 0 & 1
	\end{pmatrix}.
	\]
	Thus $P\cong U_4(\F_2)$, and hence $H^1(P,\Q/\Z)=\F_2\cdot u_{12}\oplus \F_2\cdot u_{23}\oplus \F_2\cdot u_{34}$. We conclude by \Cref{cor-comes-from-u5}.

	(xi) If $A=J_3(0)\oplus J_2(1)$, then
	\[G_A=
	\begin{pmatrix}
		1 & a & b & 0 & 0 \\
		0 & 1 & a & 0 & 0 \\
		0 & 0 & 1 & 0 & 0 \\
		0 & 0 & 0 & 1 & c \\
		0 & 0 & 0 & 0 & 1
	\end{pmatrix}.
	\]
	Consider the subgroup 
	\[K\coloneqq
	\begin{pmatrix}
		1 & a & b & 0 & 0 \\
		0 & 1 & a & 0 & 0 \\
		0 & 0 & 1 & d & e \\
		0 & 0 & 0 & 1 & c \\
		0 & 0 & 0 & 0 & 1
	\end{pmatrix}.
	\]
	Let $\tau\coloneqq I+E_{35}$ and $\sigma\coloneqq I+E_{34}$. We have $K=\ang{\sigma,\tau}\rtimes G_A$, and hence all characters of $G_A$ extend to $K$. We have $N_{\tau}(A)=E_{25}+E_{35}$ and $N_\sigma(E_{25}+E_{35})=0$, so that $N_{K/G_A}(A)=N_\sigma(N_\tau(A))=0$. By the projection formula, for all $\chi\in H^2(G_A,\Z)$, letting $\chi'\in H^2(K,\Z)$ be an element restricting to $\chi$, we have
	\[\on{cor}^{G_A}_G(A\cup\chi)=\on{cor}^K_G(\on{cor}^{G_A}_K(A\cup\chi))=\on{cor}^K_G(N_{K/G_A}(A)\cup \chi')=0.\] 
	
	(xii) If $A=J_3(0)\oplus J_1(1)\oplus J_1(1)$, then
	\[G_A=
	\begin{pmatrix}
		1 & a & b & 0 & 0 \\
		0 & 1 & a & 0 & 0 \\
		0 & 0 & 1 & 0 & 0 \\
		0 & 0 & 0 & c & d \\
		0 & 0 & 0 & e & f
	\end{pmatrix}.
	\]
	A $2$-Sylow subgroup of $G_A$ is 
	\[P\coloneqq
	\begin{pmatrix}
		1 & a & b & 0 & 0 \\
		0 & 1 & a & 0 & 0 \\
		0 & 0 & 1 & 0 & 0 \\
		0 & 0 & 0 & 1 & d \\
		0 & 0 & 0 & 0 & 1
	\end{pmatrix}.
	\]
	By \Cref{sylow}, it suffices to show that $\on{cor}^P_G(A\cup\chi)=0$ for all $\chi\in H^2(P,\Z)$. We conclude as in (xi).		
	
	(xiii) If $A=J_2(0)\oplus J_1(0)\oplus J_2(1)$, then
	\[G_A=
	\begin{pmatrix}
		1 & a & b & 0 & 0 \\
		0 & 1 & 0 & 0 & 0 \\
		0 & c & 1 & 0 & 0 \\
		0 & 0 & 0 & 1 & d \\
		0 & 0 & 0 & 0 & 1
	\end{pmatrix}.
	\]
	We conjugate $A$ by $\Pi(23)$. Then $G_A$ is replaced by
	\[G_A=
	\begin{pmatrix}
		1 & b & a & 0 & 0 \\
		0 & 1 & c & 0 & 0 \\
		0 & 0 & 1 & 0 & 0 \\
		0 & 0 & 0 & 1 & d \\
		0 & 0 & 0 & 0 & 1
	\end{pmatrix}.
	\]
	We have $H^1(G_A,\Q/\Z)=\F_2\cdot u_{12}\oplus \F_2\cdot u_{23}\oplus \F_2\cdot u_{45}$. We conclude by \Cref{cor-comes-from-u5}.

	(xiv) If $A=J_2(0)\oplus J_1(0)\oplus J_1(1)\oplus J_1(1)$, then
	\[G_A=
	\begin{pmatrix}
		1 & b & c & 0 & 0 \\
		0 & 1 & 0 & 0 & 0 \\
		0 & d & 1 & 0 & 0 \\
		0 & 0 & 0 & e & f \\
		0 & 0 & 0 & g & h
	\end{pmatrix}.
	\]
	We conjugate $A$ by $\Pi(23)$. Then $G_A$ becomes
	\[
	G_A=\begin{pmatrix}
		1 & c & b & 0 & 0 \\
		0 & 1 & d & 0 & 0 \\
		0 & 0 & 1 & 0 & 0 \\
		0 & 0 & 0 & e & f \\
		0 & 0 & 0 & g & h
	\end{pmatrix}.
	\]
	A $2$-Sylow subgroup of $G_A$ is
	\[P\coloneqq \begin{pmatrix}
		1 & c & b & 0 & 0 \\
		0 & 1 & d & 0 & 0 \\
		0 & 0 & 1 & 0 & 0 \\
		0 & 0 & 0 & 1 & f \\
		0 & 0 & 0 & 0 & 1
	\end{pmatrix}.\]
	We have $H^1(P,\Q/\Z)=\F_2\cdot u_{12}\oplus\F_2\cdot u_{23}\oplus\F_2\cdot u_{45}$. We conclude by \Cref{sylow} and \Cref{cor-comes-from-u5}.
	
	(xv) If $A=J_1(0)\oplus J_1(0)\oplus J_1(0)\oplus J_2(1)$, then
	\[G_A=
	\begin{pmatrix}
		a & b & c & 0 & 0 \\
		d & e & f & 0 & 0 \\
		g & h & i & 0 & 0 \\
		0 & 0 & 0 & 1 & j \\
		0 & 0 & 0 & 0 & 1
	\end{pmatrix}.
	\]
	A $2$-Sylow subgroup of $G_A$ is 
	\[
	P \coloneqq \begin{pmatrix}
		1 & a & c & 0 & 0 \\
		0 & 1 & b & 0 & 0 \\
		0 & 0 & 1 & 0 & 0 \\
		0 & 0 & 0 & 1 & d \\
		0 & 0 & 0 & 0 & 1
	\end{pmatrix}.
	\]
	Thus $H^1(P,\Q/\Z)=\F_2\cdot u_{12}\oplus \F_2\cdot u_{23}\oplus \F_2\cdot u_{45}$, and the conclusion follows from \Cref{sylow} and \Cref{cor-comes-from-u5}.
\end{proof}

\begin{lemma}\label{lemma-other}
	Let $A\in M$ be such that $p_A(x)$ is divisible by a square and does not split as a product of linear factors over $\F_2$. Then $\varphi_{G_A}(A\otimes\chi)=0$ for all $\chi\in H^2(G_A,\Z)$.
\end{lemma}

\begin{proof}
	By \Cref{identity-5x5}, we may assume that the trace of $A$ is equal to $0$. Write $p_A(x)=p_1(x)p_2(x)^2$, where $p_1(x)$ is square-free. Then $\on{deg}(p_2(x))\in\{1,2\}$, and hence \[p_2(x)\in \{x^2,(x+1)^2, x^2(x+1)^2,(x^2+x+1)^2,x^4\}.\]
	We exclude $x^2(x+1)^2$ and $x^4$ because by assumption $p_A(x)$ does not split over $\F_2$.
	Thus
	\[p_2(x)\in \{x^2,(x+1)^2,(x^2+x+1)^2\}.\]
	Since the trace of $A$ is zero, the sum of the roots of $p_A(x)$ in $\cl{\F}_2$ is equal to zero. As each root of $p_2(x)^2$ in $\cl{\F}_2$ has even multiplicity, we deduce that the sum of the roots of $p_1(x)$ in $\cl{\F}_2$ must be equal to $0$, so that 
    \[p_1(x)=x^d+a_{d-2}x^{d-2}+\cdots+a_1x+a_0.\]
	
	Therefore, if $p_2(x)=(x^2+x+1)^2$, then $p_1(x)=x$. If $p_2(x)\in \{x^2,(x+1)^2\}$, then $p_1(x)=x^3+a_1x+a_0$ for some $a_i\in\F_2$, but $x^3$ and $x^3+x=(x+1)x^2$ must be excluded because by assumption $p_A(x)$ does not split over $\F_2$, and hence $p_1(x)$ belongs to $\{x^3+x+1,(x+1)(x^2+x+1)\}$ in this case. All in all, the possibilities for $p_A(x)$ are
	\[(x^3+x+1)x^2,\quad (x^3+x+1)(x+1)^2,\quad (x^2+x+1)^2x,\] \[(x^2+x+1)x^2(x+1),\quad (x^2+x+1)(x+1)^3.\]
	We now prove \Cref{lemma-other} by a case-by-case analysis.
    
	(i) If $p_A(x)=(x^3+x+1)x^2$ and $q_A(x)=(x^3+x+1)x$, then up to conjugation
	\[A=\begin{bmatrix}
		0 & 0 & 1 & 0 & 0 \\
		1 & 0 & 1 & 0 & 0 \\
		0 & 1 & 0 & 0 & 0 \\
		0 & 0 & 0 & 0 & 0 \\
		0 & 0 & 0 & 0 & 0
	\end{bmatrix},\qquad 
	G_A=\begin{pmatrix}
		a+c & a & b & 0 & 0 \\
		b & c & a+b & 0 & 0 \\
		a & b & c & 0 & 0 \\
		0 & 0 & 0 & d & e\\
		0 & 0 & 0 & f & g
	\end{pmatrix}.
	\]
	The top-left $3\times 3$ corner is isomorphic to $(\F_2[x]/(x^3+x+1))^\times\cong \F_8^\times\cong\Z/7\Z$. Therefore $G_A\cong \Z/7\Z\times \on{GL}_2(\F_2)$. In particular, $I+E_{45}\in G_A$ generates a $2$-Sylow subgroup of $G_A$. The conclusion follows from \Cref{sylow} and \Cref{cor-comes-from-u5}.
	
	(ii) If $p_A(x)=q_A(x)=(x^3+x+1)x^2$, then up to conjugation
	\[A=\begin{bmatrix}
		0 & 0 & 1 & 0 & 0 \\
		1 & 0 & 1 & 0 & 0 \\
		0 & 1 & 0 & 0 & 0 \\
		0 & 0 & 0 & 0 & 1 \\
		0 & 0 & 0 & 0 & 0
	\end{bmatrix},\qquad
	G_A=\begin{pmatrix}
		a+c & a & b & 0 & 0 \\
		b & c & a+b & 0 & 0 \\
		a & b & c & 0 & 0 \\
		0 & 0 & 0 & 1 & d\\
		0 & 0 & 0 & 0 & 1
	\end{pmatrix}.
	\]
	We conclude as in (i).
	
	(iii) If $p_A(x)=(x^3+x+1)(x+1)^2$ and $q_A(x)=(x^3+x+1)(x+1)$, then $G_A$ is as in (i), and we conclude as in (i).
	
	(iv) If $p_A(x)=q_A(x)=(x^3+x+1)(x+1)^2$, then $G_A$ is as in (ii), and we conclude as in (ii).
	
	(v) If $p_A(x)=(x^2+x+1)^2x$ and $q_A(x)=(x^2+x+1)x$, then up to conjugation
	\[A=\begin{bmatrix}
		0 & 1 & 0 & 0 & 0 \\
		1 & 1 & 0 & 0 & 0 \\
		0 & 0 & 0 & 1 & 0 \\
		0 & 0 & 1 & 1 & 0 \\
		0 & 0 & 0 & 0 & 0
	\end{bmatrix},\qquad 
	G_A=\begin{pmatrix}
		a+b & a & c+d & c & 0 \\
		a & b & c & d & 0\\
		e+f & e & g+h & g & 0 \\
		e & f & g & h & 0 \\
		0 & 0 & 0 & 0 & 1
	\end{pmatrix}.
	\]
	The subring
	\[
	\begin{bmatrix}
		a+b & a \\
		a & b
	\end{bmatrix}\subset M_2(\F_2)
	\]
	is isomorphic to $\F_4$. In particular, it is commutative, and its unit group is cyclic of order $3$. Let $z\in \F_4$ be such that $z^2+z+1$, so that $\F_4=\F_2\cdot 1\oplus \F_2\cdot z$ as an $\F_2$-vector space. This identification yields an inclusion $\on{GL}_2(\F_4)\hookrightarrow \on{GL}_4(\F_2)$ with image $G_A$, where we also identify $G_A$ with its image under the injective homomorphism $G_A\hookrightarrow \on{GL}_4(\F_2)$ given by the top-left $4\times 4$ square. Since the natural $\on{GL}_2(\F_4)$-action on $\F_4^2\setminus\{0\}$ is transitive, $G_A$ acts transitively on $\F_2^4\setminus\{0\}$.
	The $G_A$-stabilizer of $e_2\in \F_2^4$ is 
	\[
	S\coloneqq \begin{pmatrix}
		1 & 0 & c+d & c & 0 \\
		0 & 1 & c & d & 0\\
		0 & 0 & g+h & g & 0 \\
		0 & 0 & g & h & 0 \\
		0 & 0 & 0 & 0 & 1
	\end{pmatrix}.
	\]
	As $[G_A:S]=|\F_2^4\setminus\{0\}|=2^4-1$ is odd, a $2$-Sylow subgroup of $S$ is also a $2$-Sylow subgroup of $G_A$. Therefore a $2$-Sylow subgroup of $G_A$ is
	\[
	P\coloneqq \begin{pmatrix}
		1 & 0 & c+d & c & 0 \\
		0 & 1 & c & d & 0\\
		0 & 0 & 1 & 0 & 0 \\
		0 & 0 & 0 & 1 & 0 \\
		0 & 0 & 0 & 0 & 1
	\end{pmatrix}.
	\]
	By \Cref{sylow}, it suffices to show that $\varphi_P=0$. We have $P\cong \Z/2\Z\times\Z/2\Z$, and hence $H^1(P,\Q/\Z)=\F_2\cdot u_{13}\oplus \F_2\cdot u_{23}$. The conclusion follows from \Cref{cor-comes-from-u5} and \Cref{cor-other-subgroups}.
	
	(vi) If $p_A(x)=q_A(x)=(x^2+x+1)^2x$, then up to conjugation
	\[A=\begin{bmatrix}
		0 & 1 & 1 & 0 & 0 \\
		1 & 1 & 0 & 1 & 0 \\
		0 & 0 & 0 & 1 & 0 \\
		0 & 0 & 1 & 1 & 0 \\
		0 & 0 & 0 & 0 & 0
	\end{bmatrix},\qquad
	G_A=\begin{pmatrix}
		a+b & a & c+d & c & 0 \\
		a & b & c & d & 0 \\
		0 & 0 & a+b & a & 0 \\
		0 & 0 & a & b & 0 \\
		0 & 0 & 0 & 0 & 1
	\end{pmatrix}.\]
	The unique $2$-Sylow subgroup of $G_A$ is
	\[P\coloneqq\begin{pmatrix}
		1 & 0 & c+d & c & 0 \\
		0 & 1 & c & d & 0 \\
		0 & 0 & 1 & 0 & 0 \\
		0 & 0 & 0 & 1 & 0 \\
		0 & 0 & 0 & 0 & 1
	\end{pmatrix}.\]
	We conclude as in (v). Indeed, by \Cref{sylow}, it suffices to show that $\varphi_P=0$. We have $H^1(P,\Q/\Z)=\F_2\cdot u_{13}\oplus \F_2\cdot u_{23}$, and the conclusion follows from \Cref{cor-comes-from-u5} and \Cref{cor-other-subgroups}.
	
	(vii) If $p_A(x)=(x^2+x+1)x^2(x+1)$ and $q_A(x)=(x^2+x+1)x(x+1)$, then up to conjugation
	\[
	A=\begin{bmatrix}
		0 & 1 & 0 & 0 & 0 \\
		1 & 1 & 0 & 0 & 0 \\
		0 & 0 & 0 & 0 & 0 \\
		0 & 0 & 0 & 0 & 0 \\
		0 & 0 & 0 & 0 & 1
	\end{bmatrix},\qquad G_A=\begin{pmatrix}
		a+b & a & 0 & 0 & 0 \\
		a & b & 0 & 0 & 0 \\
		0 & 0 & c & d & 0 \\
		0 & 0 & e & g & 0 \\
		0 & 0 & 0 & 0 & 1
	\end{pmatrix}.
	\]
	Then $G_A\cong \Z/3\Z\times \on{GL}_2(\F_2)$, and the unique $2$-Sylow subgroup of $G_A$ is
	\[P\coloneqq 
	\begin{pmatrix}
		1 & 0 & 0 & 0 & 0 \\
		0 & 1 & 0 & 0 & 0 \\
		0 & 0 & 1 & d & 0 \\
		0 & 0 & 0 & 1 & 0 \\
		0 & 0 & 0 & 0 & 1
	\end{pmatrix}.
	\]
	We conclude by \Cref{sylow} and \Cref{cor-comes-from-u5}.
	
	(viii) If $p_A(x)=q_A(x)=(x^2+x+1)x^2(x+1)$, then up to conjugation
	\[
	A=\begin{bmatrix}
		0 & 1 & 0 & 0 & 0 \\
		1 & 1 & 0 & 0 & 0 \\
		0 & 0 & 0 & 1 & 0 \\
		0 & 0 & 0 & 0 & 0 \\
		0 & 0 & 0 & 0 & 1
	\end{bmatrix},\qquad 
	G_A=\begin{pmatrix}
		a+b & a & 0 & 0 & 0 \\
		a & b & 0 & 0 & 0 \\
		0 & 0 & 1 & c & 0 \\
		0 & 0 & 0 & 1 & 0 \\
		0 & 0 & 0 & 0 & 1
	\end{pmatrix}.
	\]
	Thus $G_A\cong \Z/3\Z\times\Z/2\Z$ and the unique $2$-Sylow subgroup of $G_A$ is generated by $I+E_{34}$. We conclude by \Cref{sylow} and \Cref{cor-comes-from-u5}.
	
	(ix) If $p_A(x)=(x^2+x+1)(x+1)^3$ and $q_A(x)=(x^2+x+1)(x+1)$, then up to conjugation
	\[A=\begin{bmatrix}
		0 & 1 & 0 & 0 & 0 \\
		1 & 1 & 0 & 0 & 0 \\
		0 & 0 & 1 & 0 & 0 \\
		0 & 0 & 0 & 1 & 0 \\
		0 & 0 & 0 & 0 & 1 
	\end{bmatrix},\qquad
	G_A=\begin{pmatrix}
		a+b & a & 0 & 0 & 0 \\
		a & b & 0 & 0 & 0 \\
		0 & 0 & c & d & e \\
		0 & 0 & f & g & h \\
		0 & 0 & i & j & k
	\end{pmatrix}.
	\]
	Thus $G_A\cong \Z/3\Z \times \on{GL}_3(\F_2)$, and hence $H^2(G_A,\Z)[2]=0$.
	
	(x) If $p_A(x)=(x^2+x+1)(x+1)^3$ and $q_A(x)=(x^2+x+1)(x+1)^2$, then up to conjugation
	\[A=\begin{bmatrix}
		0 & 1 & 0 & 0 & 0 \\
		1 & 1 & 0 & 0 & 0 \\
		0 & 0 & 1 & 1 & 0 \\
		0 & 0 & 0 & 1 & 0 \\
		0 & 0 & 0 & 0 & 1 
	\end{bmatrix},\qquad G_A=\begin{pmatrix}
		a+b & a & 0 & 0 & 0 \\
		a & b & 0 & 0 & 0 \\
		0 & 0 & 1 & c & d \\
		0 & 0 & 0 & 1 & 0 \\
		0 & 0 & 0 & e & 1
	\end{pmatrix}.
	\]
	We conjugate $A$ by $\Pi(45)$. Then $G_A$ is replaced by
	\[
	G_A=\begin{pmatrix}
		a+b & a & 0 & 0 & 0 \\
		a & b & 0 & 0 & 0 \\
		0 & 0 & 1 & d & c \\
		0 & 0 & 0 & 1 & e \\
		0 & 0 & 0 & 0 & 1
	\end{pmatrix}.
	\]
	Then $G_A\cong \Z/3\Z\times U_3(\F_2)$, and the unique $2$-Sylow subgroup of $G_A$ is
	\[
	P\coloneqq \begin{pmatrix}
		1 & 0 & 0 & 0 & 0 \\
		0 & 1 & 0 & 0 & 0 \\
		0 & 0 & 1 & d & c \\
		0 & 0 & 0 & 1 & e \\
		0 & 0 & 0 & 0 & 1
	\end{pmatrix}.
	\]
	Thus $P\cong U_3(\F_2)$ and in particular $H^1(P,\Q/\Z)=\F_2\cdot u_{34}\oplus \F_2\cdot u_{45}$. We conclude by \Cref{sylow} and \Cref{cor-comes-from-u5}.
	
	(xi) If $p_A(x)=q_A(x)=(x^2+x+1)(x+1)^3$, then up to conjugation
	\[A=\begin{bmatrix}
		0 & 1 & 0 & 0 & 0 \\
		1 & 1 & 0 & 0 & 0 \\
		0 & 0 & 1 & 1 & 0 \\
		0 & 0 & 0 & 1 & 1 \\
		0 & 0 & 0 & 0 & 1 
	\end{bmatrix},\qquad G_A=\begin{pmatrix}
		a+b & a & 0 & 0 & 0 \\
		a & b & 0 & 0 & 0 \\
		0 & 0 & 1 & c & d \\
		0 & 0 & 0 & 1 & c \\
		0 & 0 & 0 & 0 & 1
	\end{pmatrix}.\]
	Then $G_A\cong \Z/3\Z\times\Z/4\Z$, and the unique $2$-Sylow subgroup of $G_A$ is
	\[
	P\coloneqq \begin{pmatrix}
		1 & 0 & 0 & 0 & 0 \\
		0 & 1 & 0 & 0 & 0 \\
		0 & 0 & 1 & c & d \\
		0 & 0 & 0 & 1 & c \\
		0 & 0 & 0 & 0 & 1
	\end{pmatrix}.
	\]
	By \Cref{sylow}, it suffices to show that $\on{cor}^P_G(A\cup\chi)=0$ for all $\chi\in H^2(P,\Z)$. Let 
	\[K\coloneqq \begin{pmatrix}
		1 & e & 0 & 0 & 0 \\
		0 & 1 & 0 & 0 & 0 \\
		0 & 0 & 1 & c & d \\
		0 & 0 & 0 & 1 & c \\
		0 & 0 & 0 & 0 & 1
	\end{pmatrix},\]
	and let $\sigma\coloneqq I+E_{12}\in K$, so that $K=\ang{\sigma,P}\cong \Z/2\Z\times P$. 
	Every character of $P$ extends to $K$. Let $A'\coloneqq N_{\sigma}(A)=E_{11}+E_{22}\in M^K$. By the projection formula, for every $\chi\in H^2(P,\Z)$, letting $\chi'\in H^2(K,\Z)$ be a class restricting to $\chi$, we have
	\[\on{cor}^P_G(A\cup\chi)=\on{cor}^K_G(A'\cup\chi').\]
	Because $A'\in M^K$, we have $K\subset G_{A'}$, and so by the projection formula
	\[\on{cor}^K_G(A'\cup\chi')=\on{cor}^{G_{A'}}_G(A'\cup N_{G_{A'}/K}(\chi')).\]
	Since $A'$ is diagonal, the conclusion follows from \Cref{lemma-diagonalizable}.
\end{proof}

\subsection{Restriction to a Klein subgroup}
Let
\[
	S \coloneqq
	\begin{bmatrix}
		1 & 0 \\
		0 & 1
	\end{bmatrix}, 
	\quad
	T \coloneqq 
	\begin{bmatrix}
		0 & 1 \\
		1 & 1
	\end{bmatrix}
	\]
be matrices in $M_2(\mathbb{F}_2)$. Let $n\geq 4$. We will write $n \times n$ matrices as $3 \times 3$ matrices according to the partition $n = 2 + 2 + (n - 4)$. The matrices
\[
\sigma \coloneqq 
\begin{pmatrix}
	I & S & 0 \\
	0 & I & 0 \\
	0 & 0 & I
\end{pmatrix}, 
\quad
\tau \coloneqq
\begin{pmatrix}
	I & T & 0 \\
	0 & I & 0 \\
	0 & 0 & I
\end{pmatrix}
\]
commute and generate a Klein subgroup $Z\subset \on{GL}_n(\F_2)$.

\begin{prop}\label{restrict-klein-f2}
	For every $n\geq 4$, the class of $\on{GLift}(\F_2,n)$ restricts to a non-trivial class in $H^2(Z,M_n(\mathbb{F}_2))$, where $Z=\ang{\sigma,\tau}\subset \on{GL}_n(\F_2)$ is the Klein subgroup defined above.
\end{prop}

\begin{proof}
	Let $M\coloneqq M_n(\F_2)$. Let
\[
	\tilde{S} \coloneqq
	\begin{bmatrix}
		1 & 0 \\
		0 & 1
	\end{bmatrix}, 
	\qquad
	\tilde{T} \coloneqq 
	\begin{bmatrix}
		0 & 1 \\
		1 & 1
	\end{bmatrix}
	\]
be matrices in $M_2(\Z/4\Z)$, and define
\[
\tilde{\sigma} \coloneqq 
\begin{pmatrix}
	I & \tilde{S} & 0 \\
	0 & I & 0 \\
	0 & 0 & I
\end{pmatrix}, 
\qquad
\tilde{\tau} \coloneqq
\begin{pmatrix}
	I & \tilde{T} & 0 \\
	0 & I & 0 \\
	0 & 0 & I
\end{pmatrix}
\]
in $\on{GL}_n(\Z/4\Z)$. Then $\tilde{\sigma}$ and $\tilde{\tau}$ commute, $\tilde{\sigma}^{-2}=I+2s$ and $\tilde{\tau}^2=I+2t$, where 
	\[
	s \coloneqq 
	\begin{pmatrix}
		0 & S & 0 \\
		0 & 0 & 0 \\
		0 & 0 & 0
	\end{pmatrix}, 
	\qquad
	t \coloneqq  
	\begin{pmatrix}
		0 & T & 0 \\
		0 & 0 & 0 \\
		0 & 0 & 0
	\end{pmatrix}.
	\]
    Suppose by contradiction that $\on{GLift}(\F_2,n)$ restricts to the trivial class in $H^2(Z,M)$. Then, by \Cref{klein-cohomology}, there are $U, V \in M$ such that:
	\[
	N_\sigma(U) = s, \quad N_\tau(V) = t, \quad N_\tau(U) = N_\sigma(V).
	\]
	We have:
	\[
	N_\sigma(U) = U + \sigma U \sigma = 
	\begin{bmatrix}
		* & U_{11} + U_{21} + U_{22} & * \\
		* & U_{21} & * \\
		* & * & *
	\end{bmatrix},
	\]
    and
    \[
	N_\tau(V) = V + \tau U \tau = 
	\begin{bmatrix}
		* & V_{11}T + TV_{21}T + TV_{22} & * \\
		* & TV_{21} & * \\
		* & * & *
	\end{bmatrix},
	\]
	hence the equation $N_\sigma(U)=s$ implies that $U_{21} = 0$ and
	\begin{equation} \label{eq1}
		U_{11} + U_{22} = I.
	\end{equation}
	Similarly, the equation $N_\tau(V)=t$ implies that $V_{21} = 0$ and
	\begin{equation} \label{eq2}
		V_{11}T + TV_{22} = T.
	\end{equation}
	The equation $N_\tau(U) = N_\sigma(V)$ implies
	\begin{equation} \label{eq3}
		U_{11}T + TU_{22} = V_{11} + V_{22}.
	\end{equation}
	Plugging in $U_{22} = I + U_{11}$ from \eqref{eq1} into \eqref{eq3} and then $V_{22}$ from \eqref{eq3} into \eqref{eq2}, we get:
	\[
	T(V_{11} + TU_{11}) + (V_{11} + TU_{11})T = T + T^2 = I.
	\]
	Note that the equation $TX + XT = I$ has no solutions in $M_2(\mathbb{F}_2)$, contradicting the existence of $U$ and $V$, as desired.
\end{proof}

\subsection{The case when \texorpdfstring{$A$}{A} is conjugate to a Jordan block}

\begin{prop}\label{lemma-jordan-block}
	Suppose that $A\in M$ is conjugate to a $5\times 5$ Jordan block. Then the class of $\on{GLift}(\F_2,5)$ is not in the image of $\varphi_{G_A}$.    
\end{prop}

\begin{proof}
	By \Cref{identity-5x5}, we may assume that the trace of $A$ is $0$, and hence, up to conjugation, that $A$ is the nilpotent $5\times 5$ Jordan block $N=J_5(0)$. We have
	\[G_A=
	\begin{pmatrix}
		1 & a & b & c & d \\
		0 & 1 & a & b & c \\
		0 & 0 & 1 & a & b \\
		0 & 0 & 0 & 1 & a \\
		0 & 0 & 0 & 0 & 1
	\end{pmatrix}.
	\]
	Thus $G_A\cong (\F_2[x]/(x^5))^\times\cong \Z/8\Z\times\Z/2\Z$, where the factor $\Z/8\Z$ is generated by $I+N$ and the factor $\Z/2\Z$ is generated by $I+N^3=I+E_{14}+E_{25}$.
	
	Define $u,v\in H^1(G_A,\Q/\Z)$ by 
    \[u(I+N)=0,\quad u(I+N^3)=1/2,\quad v(I+N)=1/8,\quad v(I+N^3)=1/2,\]
    and let $\chi\coloneqq \partial(u)$ and $\psi\coloneqq \partial(v)$ in $H^2(G_A,\Z)$. We have 
    \[H^2(G_A,\Z)=(\Z/2\Z)\cdot \chi\oplus (\Z/8\Z)\cdot \psi.\] We first show that $\on{cor}^{G_A}_G(A\cup\chi)=0$. Consider the subgroup
	\[
	K\coloneqq\begin{pmatrix}
		1 & a & b & d & f \\
		0 & 1 & a & c & e \\
		0 & 0 & 1 & a & c \\
		0 & 0 & 0 & 1 & a \\
		0 & 0 & 0 & 0 & 1
	\end{pmatrix}.
	\]
	Then $G_A\subset K$. We claim that $u$ extends to an element of $H^1(K,\Q/\Z)$. For this, let $\cl{K}$ be the quotient of $K$ by the subgroup generated by $I+E_{13}, I+E_{14}, I+E_{15}$: 
	\[\cl{K}=\begin{pmatrix}
		1 & a & \square & \square & \square \\
		0 & 1 & a & c & e \\
		0 & 0 & 1 & a & c \\
		0 & 0 & 0 & 1 & a \\
		0 & 0 & 0 & 0 & 1
	\end{pmatrix}.\]
	Then $\cl{K}\cong \Z/4\Z\times \Z/2\Z$, where the $\Z/4\Z$ is generated by the coset of $I+N$, and the $\Z/2\Z$ is generated by the coset of $I+N^3$. It follows that we may define $\cl{u}\in H^1(\cl{K},\Q/\Z)$ by sending the coset of $I+N$ to $0$ and the coset of $I+N^3$ to $1/2$. Letting $u'\in H^1(K,\Q/\Z)$ be the composition of the quotient map $K\to \cl{K}$ and $\cl{u}$, we see that $u'$ restricts to $u$ on $G_A$, as claimed. It follows that $\chi'\coloneqq \partial(u')\in H^2(K,\Z)$ extends $\chi$.
	
	Let $\sigma\coloneqq I+E_{14}$ and $\tau\coloneqq I+E_{13}$. Then $N_\sigma(A)=E_{15}$ and $N_\tau(E_{15})=0$, so that $N_{K/G_A}(A)=N_\tau(N_\sigma(A))=0$. By the projection formula, 
	\[\on{cor}^{G_A}_G(A\cup\chi)=\on{cor}^{K}_G(\on{cor}^{G_A}_K(A\cup\chi))=\on{cor}^{K}_G(N_{K/G_A}(A)\cup \chi')=0.\]
	
	It remains to show that $\on{cor}^{G_A}_G(A\cup\psi)=0$. For this, let $Z\coloneqq\ang{\sigma,\tau}\subset U$ be the Klein subgroup of \Cref{restrict-klein-f2}. 
	By \Cref{restrict-klein-f2}, it suffices to show that $(\on{res}^G_Z\circ \on{cor}^{G_A}_G)(A\cup\psi)=0$. The double coset formula reads
	\begin{equation}\label{double-coset-1}\on{res}^G_Z\circ \on{cor}^{G_A}_G=\sum_{g\in R} \on{cor}_Z^{Z\cap gG_Ag^{-1}}\circ\,  g_*\circ \on{res}^{G_A}_{G_A\cap g^{-1}Zg},
    \end{equation}
    where $R\subset G$ is a set of representatives for $Z\backslash G/G_A$. The Jordan normal form of $I+N^4$ is $I+E_{12}$, while the Jordan normal form of $\sigma,\tau,\sigma\tau$ is $I+E_{12}+E_{34}$. Thus $g(I+N^4)g^{-1}$ does not belong to $Z$, for any $g$. It follows that there are three mutually exclusive possibilities for $Z\cap gG_Ag^{-1}$: either it is trivial, or it is generated by $I+N^3$, or it is generated by $\rho\coloneqq I+N^3+N^4$. In the first two cases, the restriction of $v$ to  $Z\cap gG_Ag^{-1}$ is zero, and hence the term in (\ref{double-coset-1}) corresponding to $g$ is zero. Thus (\ref{double-coset-1}) reduces to
	\begin{equation}\label{double-coset-2}\on{res}^G_Z\circ \on{cor}^{G_A}_G=\sum_{g\in S} \on{cor}_Z^{Z\cap gG_Ag^{-1}}\circ\,  g_*\circ \on{res}^{G_A}_{G_A\cap g^{-1}Zg},
    \end{equation}
    where $S\subset R$ is the subset of those $g$ such that $g \rho g^{-1}\in Z$. We have $S=S_\sigma\coprod S_\tau\coprod S_{\sigma\tau}$, where by definition $g$ belongs to $S_\sigma$ (resp. $S_\tau$, $S_{\sigma\tau}$) if and only if $g\rho g^{-1}$ is equal to $\sigma$ (resp. $\tau$, $\sigma\tau$). 
    
    For all $g\in S_\sigma$, the subgroup $Z\cap gG_Ag^{-1}$ is equal to $\ang{\sigma}$. Moreover, $g_*(v)$ is the non-trivial element in $H^1(\ang{\sigma},\Q/\Z)$, and hence $g_*(\psi)=\partial(g_*(v))$ is the unique non-trivial element in $H^2(\ang{\sigma},\Z)$. Let $\psi_\sigma\in H^2(Z,\Z)$ which extends $g_*(\psi)$ for $g\in S_\sigma$. By the projection formula, for all $g\in S_\sigma$ we have
	\begin{align*}(\on{cor}_Z^{Z\cap gG_Ag^{-1}}\circ\,  g_*\circ \on{res}^{G_A}_{G_A\cap g^{-1}Zg})(A\cup\psi)&=(\on{cor}_Z^{\ang{\sigma}}\circ\,  g_*\circ \on{res}^{G_A}_{\ang{\rho}})(A\cup\psi) \\
		&=\on{cor}_Z^{\ang{\sigma}}(g_*(A)\cup g_*(\psi))\\
		&=N_{Z/\ang{\sigma}}(g_*(A))\cup \psi_\sigma \\
		&=(gag^{-1}+\tau gag^{-1}\tau^{-1})\cup \psi_\sigma.
	\end{align*}
    Therefore
    \[\sum_{g\in S_\sigma}(\on{cor}_Z^{Z\cap gG_Ag^{-1}}\circ\,  g_*\circ \on{res}^{G_A}_{G_A\cap g^{-1}Zg})(A\cup\psi)=\left(\sum_{g\in S_\sigma}(gag^{-1}+\tau gag^{-1}\tau^{-1})\right)\cup \psi_\sigma.\]
	Similarly,
	\[\sum_{g\in S_\tau}(\on{cor}_Z^{Z\cap gG_Ag^{-1}}\circ\,  g_*\circ \on{res}^{G_A}_{G_A\cap g^{-1}Zg})(A\cup\psi)=\left(\sum_{g\in S_\tau}(gag^{-1}+\sigma gag^{-1}\sigma^{-1})\right)\cup \psi_\tau,\]
	\[\sum_{g\in S_{\sigma\tau}}(\on{cor}_Z^{Z\cap gG_Ag^{-1}}\circ\,  g_*\circ \on{res}^{G_A}_{G_A\cap g^{-1}Zg})(A\cup\psi)=\left(\sum_{g\in S_{\sigma\tau}}(gag^{-1}+\sigma gag^{-1}\sigma^{-1})\right)\cup \psi_{\sigma\tau},\]
    where $\varphi_\tau$ (resp. $\varphi_{\sigma\tau}$) is an element of $H^2(Z,\Z)$ extending $g_*(\psi)$ for all $g\in S_\tau$ (resp. $g\in S_{\sigma\tau}$). In view of (\ref{double-coset-2}), the proof will be complete once we show that the three sums
	\[\sum_{g\in S_\sigma}(gag^{-1}+\tau gag^{-1}\tau^{-1}),\;\, \sum_{g\in S_{\tau}}(gag^{-1}+\sigma gag^{-1}\sigma^{-1}),\;\, \sum_{g\in S_{\sigma\tau}}(gag^{-1}+\sigma gag^{-1}\sigma^{-1})\]
	are zero. 
    
    For all $g\in S_\sigma$, we have $g \rho = \sigma g$, hence $\sigma g G_A = g \rho G_A = g G_A$, and so
	\[ZgG_A=gG_A\cup \sigma gG_A\cup \tau gG_A\cup\sigma\tau gG_A=gG_A\cup\tau gG_A=\ang{\tau}gG_A.\]
	In other words, 		
	\[Z \backslash (ZS_\sigma G_A)/ G_A = \ang{\tau}\backslash (S_\sigma G_A)/ G_A.\]
	Note that $\tau$ acts without fixed points on $(S_\sigma G_A)/ G_A$. Indeed, suppose that $\tau g G_A=g G_A$ for some $g\in S_g$. Then $g^{-1}\tau g\in G_A$. We also have $g^{-1}\sigma g=\rho\in G_A$, and hence $g^{-1}Zg\subset G_A$. As $Z$ and the $2$-torsion subgroup $G_A[2]\subset G_A$ have the same order, equal to $4$, this implies that $g^{-1}Zg=G_A[2]$, contradicting the fact that $I+N^4$ is not conjugate to any element of $Z$.	We obtain
	\[\sum_{g\in S_\sigma}(gag^{-1}+\tau gag^{-1}\tau^{-1})=\sum_g gag^{-1},\]
	where the second sum is taken over a set of representatives $g$ of the cosets in $(S_\sigma G_A)/ G_A$. Let $C\subset G$ be the centralizer of $\rho$. Observe that $G_A\subset C$: indeed, a matrix commuting with $A=I+N$ commutes with any polynomial in $N$ such as $\rho$. Moreover, $S_\sigma G_A=g_0C$ for some $g_\sigma\in S_\sigma$. It follows that the above sum is conjugate via $g_0$ to $N_{C/G_A}(A)$. The same argument shows that the second and third sums are conjugate to $N_{C/G_A}(A)$ via appropriate $g_\tau\in S_\tau$ and $g_{\sigma\tau}\in S_{\sigma\tau}$, respectively. It remains to show that $N_{C/G_A}(A)=0$. Consider again the subgroup
	\[K\coloneqq \begin{pmatrix}
		1 & a & b & d & f \\
		0 & 1 & a & c & e \\
		0 & 0 & 1 & a & c \\
		0 & 0 & 0 & 1 & a \\
		0 & 0 & 0 & 0 & 1
	\end{pmatrix}.
	\]
	Then $G_A\subset K\subset C$, and hence it suffices to show that $N_{K/G_A}(A)=0$. Let $\mu\coloneqq I+E_{14}$ and $\nu\coloneqq I+E_{13}$. Then $G_A$ is normal in $K$ and $K/G_A$ is a Klein group generated by the cosets of $\mu$ and $\nu$. We have $N_\mu(A)=I+E_{15}$ and $N_{\nu}(I+E_{15})=0$, so that $N_{K/G_A}=N_\nu(N_{\mu}(A))=0$, as desired.
\end{proof}

\subsection{End of Proof of Theorem \ref{mainthm}}

\begin{proof}[Proof of \Cref{mainthm} when $|k|=2$ and $n\geq 5$]
	By \Cref{blift-implies-glift}, it suffices to show that $\on{GLift}(\F_2,n)$ is not negligible over $F$, and by \Cref{reduce-dimension}, we may assume that $n=5$. By \Cref{basic-lemma}(3), we may also assume that $F$ contains all primitive roots of unity of $2$-power order. 
    
    By \Cref{all-phi-but-one-vanish} and \Cref{lemma-jordan-block}, the class of $\on{GLift}(\F_2,5)$ does not belong to the subgroup of $H^2(\on{GL}_5(\F_2),M_5(2))$ generated by the images of the maps $\varphi_H$, where $H$ ranges over all subgroups of $\on{GL}_5(\F_2)$.  Now \Cref{thm-negligible} implies that the class of $\on{GLift}(\F_2,5)$ is not negligible over $F$.
\end{proof}

\subsection{Explicit non-liftable Galois representations}\label{explicit}
    Let $F$ be a field. Let $H$ be a finite group, let $V$ be a faithful finite-dimensional $F$-linear representation of $H$ over $F$. We view $V$ as an affine space over $F$, we let $F(V)$ be the function field of $V$, and we let $F(V)^H$ be the $H$-fixed subfield. The field extension $F(V)/F(V)^H$ is Galois with Galois group $H$. A choice of separable closure of $F(V)^H$ containing $F(V)$ determines a surjective homomorphism $\rho\colon \Gamma_{F(V)^H}\to H$. We say that a pair $(K,\rho)$, where $K/F$ is a field extension and $\rho\colon \Gamma_K\to H$ is a homomorphism, is \emph{generic for $H$ over $F$} if there exists a faithful finite-dimensional $F$-linear representation $V$ of $H$ such that $K=F(V)^H$ and $\rho$ is induced by the $H$-Galois extension $F(V)/F(V)^H$. Of course, a generic pair for $H$ over $F$ always exists.
    
    \begin{prop}\label{connection}
		Let $H$ be a finite group, let $A$ be a $H$-module, let $(K,\rho)$ be a generic pair for $H$ over $F$, and consider a group extension (\ref{general-lift}). The class of (\ref{general-lift}) is negligible over $F$ if and only if $\rho$ lifts to $G$.
	\end{prop}
	
	\begin{proof}
	See \cite[Proposition 2.1]{gherman2022negligible}.
	\end{proof}

For all positive integers $n$ and fields $k$ of characteristic $p>0$ such that the class of $\on{GLift}(k,n)$ is not negligible over $F$, using \Cref{connection} we now exhibit field extensions $K/F$ and continuous homomorphisms $\rho\colon \Gamma_K\to \on{GL}_n(k)$ which do not lift to $\Gamma_K\to\on{GL}_n(W_2(k))$. Indeed, one may take a generic pair $(K,\rho)$ for $H$ over $F$, where $H$ is the finite subgroup of $\on{GL}_n(k)$ given below. We may assume that $\on{char}(F)\neq p$, since $\on{GLift}(k,n)$ is otherwise negligible over $F$.

\begin{itemize}
    \item If $p>2$ and $n\geq 3$, we may take $H=\on{GL}_3(\F_p)$, embedded in the top-left $3\times 3$ block of $\on{GL}_n(k)$; see the proofs of \Cref{reduce-to-fp} and from \Cref{reduce-dimension}.
    \item If $p=2$, $|k|>2$, $n\geq 3$, we may take $H\subset \on{GL}_3(k)\subset \on{GL}_n(k)$ to be the subgroup of $\on{GL}_3(k)$ appearing in the statement of \Cref{restriction-kills-negligible}, where $\on{GL}_3(k)\subset \on{GL}_n(k)$ is the top-left $3\times 3$ block; see the proofs of \Cref{restriction-kills-negligible} and \Cref{reduce-dimension}.
    \item If $p=2$, $|k|=2$ and $n\geq 5$, we may take $H=\on{GL}_5(\F_2)$, embedded in $\on{GL}_n(\F_2)$ as the top-left $5\times 5$ corner; see the proof of \Cref{reduce-dimension}.
\end{itemize}    

\section{Splitting of \texorpdfstring{$\mathrm{Lift}(k,n)$}{Lift(k,n)}}

For completeness, we determine all cases when the sequence $\mathrm{Lift}(k,n)$ is split. 

\begin{thm}\label{lift-split}
Let $k$ be a field of characteristic $p > 0$ and let $n > 0$ be an integer. The sequence $\on{GLift}(k, n)$
is split if and only if one of the following holds:
\begin{itemize}
  \item $n = 1$;
  \item $n = 2$ and $|k|\leq 3$;
  \item $n = 3$ and $|k| = 2$.
\end{itemize}
\end{thm}

\begin{proof}
We first show that $\on{GLift}(k,n)$ splits in the cases listed above.

(i) If $n=1$, a splitting of the map $\pi\colon W_2(k)^\times \to k^\times$ is given by the Teichm\"uller lift, that is, the group homomorphism $\tau\colon k^\times \to W_2(k)^\times$ given by $\tau(x)=(x,0)$.

In all remaining cases, $k$ is finite, and hence the sequence $\on{GLift}(k,n)$ is split if and only if its restriction to the $p$-Sylow subgroup $U_n(k)\subset \on{GL}_n(k)$ is split. We will construct splittings over $U_n(k)$.

(ii) If $n = 2$ and $k = \F_2$, a splitting is given by
\[
\begin{pmatrix}
    1 & 1 \\
    0 & 1
\end{pmatrix}
\mapsto
\begin{pmatrix}
    -1 & 1 \\
    \ \ 0 & 1
\end{pmatrix}.
\]

(iii) If $n = 2$ and $k = \F_3$, a splitting is given by
\[
\begin{pmatrix}
    1 & 1 \\
    0 & 1
\end{pmatrix}
\mapsto
\begin{pmatrix}
    \ \ 1 & \ \ 1 \\
    -3 & -2
\end{pmatrix}.
\]

(iv) If $n = 3$ and $k = \F_2$, a splitting is given by
\[
\begin{pmatrix}
    1 & 1 & 0 \\
    0 & 1 & 0 \\
    0 & 0 & 1
\end{pmatrix}
\mapsto
\begin{pmatrix}
    -1 & -1 & \ \ 0 \\
    \ \ 0 & \ \ 1 & \ \ 0 \\
    \ \ 0 & \ \ 0 & -1
\end{pmatrix},\quad
\begin{pmatrix}
    1 & 0 & 0 \\
    0 & 1 & 1 \\
    0 & 0 & 1
\end{pmatrix}
\mapsto
\begin{pmatrix}
    -1 & 0 & \ \ 0 \\
    \ \ 2 & 1 & -1 \\
    \ \ 0 & 0 & -1
\end{pmatrix}.
\]
Indeed, letting $\sigma_1,\sigma_2\in \on{GL}_3(\Z/4\Z)$ be the images of $I+E_{12},I+E_{23}$, and letting $\tau\coloneqq [\sigma_1,\sigma_2]$, it suffices to check that
\[\sigma_1^2=\sigma_2^2=\tau^2=[\sigma_1,\tau]=[\sigma_2,\tau]=1,\]
which can be done by direct matrix computations. This completes the proof that $\on{GLift}(k,n)$ splits in the cases listed above. 

In order to complete the proof of \Cref{lift-split}, it remains to prove that in all other cases $\on{GLift}(k, n)$ is not split.

(1) If $n \geq 3$ and $p \geq 3$, the conclusion follows from \Cref{reduce-to-fp} and \cite[Claim 5.4]{merkurjev2024galois}. (One could replace \cite[Claim 5.4]{merkurjev2024galois} by the stronger \Cref{thm-gl3}.)

(2) If $p = 2$ and $n \geq 2$ and $|k| > 2$, see \Cref{restrict-k-bigger-f2}.

(3) If $p = 2$ and $n \geq 4$, see \Cref{restrict-klein-f2}.

(4) If $p > 3$ and $n = 2$, see \cite[Remark 5.8(1)]{merkurjev2024galois}.

(5) It remains to consider the case $p = 3$, $n = 2$ and $|k| > 3$. Choose $x, y \in k$ which are linearly independent over $\F_3$, and let $\rho\coloneqq I + xE_{12}$ and $\mu\coloneqq I + yE_{12}$. Observe that $\rho$ and $\mu$ generate a subgroup $H\cong (\Z/3\Z)^2$ of $U_2(k)$. We will show that the restriction of $\on{GLift}(k, 2)$ to $H$
is not trivial. Let $\tilde{x}\coloneqq (x,0)$ and $\tilde{y}\coloneqq (y,0)$ in $W_2(k)$, so that $\tilde{\rho}\coloneqq I+\tilde{x}E_{12}$ and $\tilde{\mu}\coloneqq I+\tilde{y}E_{12}$ are lifts of $\rho$ and $\mu$ to $\on{GL}_2(W_2(k))$, respectively. Observe that $3(z,0)=(0,z^3)=\iota(z^3)$ for all $z\in k$. Thus 
    \[\tilde{\rho}^3=I+\iota(x^3)E_{12},\qquad \tilde{\mu}^3=I+\iota(y^3)E_{12},\qquad [\tilde{\rho},\tilde{\mu}]=I.\] 
Suppose by contradiction that the restriction of $\on{GLift}(k,2)$ to $H$ splits. Then, by \Cref{klein-cohomology}, there exist $U=(u_{ij})$ and $V=(v_{ij})$ in $M_2(k)$ such that 
\[N_\rho(U)=x^3E_{12},\qquad N_\mu(V)=y^3E_{12},\qquad (\rho-1)V-(\mu-1)U=0.\]
On the other hand, a matrix computation shows that $N_\rho(U)=u_{21}x^6E_{12}$ and $N_\mu(V)=v_{21}y^6E_{12}$, and that the $(1,1)$-th entry of $(\rho-1)V-(\mu-1)U$ is equal to $x^3 v_{21} - y^3 u_{21}$. We obtain  $u_{21}=x^{-3}$ and $v_{21}=y^{-3}$, and hence $x^6=y^6$, that is, $x=\pm y$. This contradicts the fact that $x$ and $y$ are linearly independent over $\F_3$. We conclude that the restriction of $\on{GLift}(k,2)$ to $H$ does not split, as desired.
\end{proof}

\begin{rmk}
    In cases (ii)-(iv) of the proof of \Cref{lift-split}, where $k=\F_p$ for $p\in \{2,3\}$, the splittings $U_n(\F_p)\to \on{GL}_n(\F_p)$ are integral, that is, they lift to homomorphisms $U_n(\F_p)\to \on{GL}_n(\Z)$ defined by the same matrices, this time viewed as matrices with integer coefficients.
\end{rmk}

\end{document}